\documentclass[12pt]{article}
\usepackage{latexsym,amsmath,amssymb}

\newif\ifdviwin

\usepackage{enumerate}
\usepackage[latin1]{inputenc}
\usepackage[english]{babel}
\usepackage{indentfirst}
\usepackage[mathscr]{eucal}
\usepackage{amssymb,amsmath,amsfonts}
\usepackage{graphicx,epsfig}

\newif\ifdviwin

\dviwintrue

\def\cK{\mathcal{K}}
\def\cJ{\mathcal{J}}
\def\cQ{\mathcal{Q}}

\def\cF{\mathcal{F}}

\def\cB{\mathcal{B}}

\def\cL{\mathcal{L}}
\def\cU{\mathcal{U}}
\def\cG{\mathcal{G}}
\def\cH{\mathcal{H}}

\let\8=\infty \let\0=\emptyset  
\let\hat=\widehat

\let\tilde=\widetilde

\let\landa=\lambda
\let\alfa=\alpha

\let\parc=\partial

\def\landa{\lambda}

\def\flecha{\rightarrow}
\def\esiz{\langle}
\def\esde{\rangle}

\def\cte.{\mathop{\rm cte.}\nolimits}

\def\det{\mathop{\rm det}\nolimits}

\def\N{\mathbb{N}}

\def\R{\mathbb{R}}

\def\C{\mathbb{C}}
\def\D{\mathbb{D}}

\def\H{\mathbb{H}}
\def\S{\mathbb{S}}

\newcommand{\rmd}{\mathrm{d}}
\newcommand{\rmT}{\mathrm{T}}
\newcommand{\rmC}{\mathrm{C}}
\newcommand{\s}{\mathbb{S}}

\DeclareMathOperator{\re}{Re}
\DeclareMathOperator{\im}{Im}
\DeclareMathOperator{\Ric}{Ric}
\DeclareMathOperator{\dist}{dist}
\DeclareMathOperator{\ind}{Ind}
\DeclareMathOperator{\grad}{grad}

\newcommand{\nil}{\mathrm{Nil}_3}
\newcommand{\sol}{\mathrm{Sol}_3}

 \newtheorem{teo}{Theorem}[section]
 \newtheorem{pro}[teo]{Proposition}
 \newtheorem{cor}[teo]{Corollary}
 \newtheorem{lem}[teo]{Lemma}
\newtheorem{fact}[teo]{Fact}
\newtheorem{conj}[teo]{Conjecture}

\newtheorem{defin}[teo]{Definition}
\newenvironment{defi}{\begin{defin} \rm}{\end{defin}}
\newtheorem{remk}[teo]{Remark}
\newenvironment{remark}{\begin{remk} \rm}{\end{remk}}
\newtheorem{exa}[teo]{Example}
\newenvironment{eje}{\begin{exa} \rm}{\end{exa}}

 \newenvironment{proof}{\rm \trivlist \item[\hskip \labelsep{\it
      Proof}:]}{\par\nopagebreak \hfill $\Box$ \endtrivlist}

 \newenvironment{proof1}{\rm \trivlist \item[\hskip \labelsep{\it
      Proof of Theorems \ref{main} and \ref{hopfmain}}:]}{\par\nopagebreak \hfill $\Box$ \endtrivlist}

\numberwithin{equation}{section}

\begin{document}
\mbox{}\vspace{0.4cm}\mbox{}

\begin{center}
\rule{14cm}{1.5pt}\vspace{0.5cm}

{\Large \bf Existence and uniqueness of constant mean} \\ [0.3cm]{\Large \bf  curvature spheres in $\sol$}\\
\vspace{0.5cm} {\large Beno\^{\i}t Daniel$\mbox{}^a$ and Pablo Mira$\mbox{}^b$}\\ \vspace{0.3cm}
\rule{14cm}{1.5pt}
\end{center}
  \vspace{1cm}
$\mbox{}^a$ Universit\'e Paris 12, D\'epartement de Math\'ematiques, UFR des Sciences et Technologies, 61 avenue du G\'en\'eral de Gaulle, 94010 Cr\'eteil cedex, France \\ e-mail: daniel@univ-paris12.fr
\vspace{0.2cm}

\noindent $\mbox{}^b$ Departamento de Matem\'atica Aplicada y Estad\'{\i}stica,
Universidad Polit\'ecnica de Cartagena, E-30203 Cartagena, Murcia, Spain. \\
e-mail: pablo.mira@upct.es \vspace{0.2cm}

\vspace{0.2cm}



\vspace{0.3cm}

\begin{abstract}
We study the classification of immersed constant mean curvature (CMC) sphe-res in the homogeneous Riemannian $3$-manifold $\sol$, i.e., the only Thurston $3$-dimensional geometry where this problem remains open. Our main result states that, for every $H>1/\sqrt{3}$, there exists a unique (up to left translations) immersed CMC $H$ sphere $S_H$ in $\sol$ (Hopf-type theorem). Moreover, this sphere $S_H$ is embedded, and is therefore the unique (up to left translations) compact embedded CMC $H$ surface in $\sol$ (Alexandrov-type theorem). The uniqueness parts of these results are also obtained for all real numbers $H$ such that there exists a solution of the isoperimetric problem with mean curvature $H$.
\end{abstract}

\section{Introduction}\label{sec:intro}

Two fundamental results in the theory of compact constant mean curvature (CMC) surfaces are the Hopf and Alexandrov theorems. The first one \cite{hopfpaper} states that round spheres are the unique immersed CMC spheres in Euclidean space $\R^3$; the proof relies on the existence of a holomorphic quadratical differential, the so-called Hopf differential. The second one \cite{alexpaper} states that round spheres are the unique compact embedded CMC surfaces in Euclidean space $\R^3$; the proof is based on the so-called Alexandrov reflection technique, and uses the maximum principle. Hopf's theorem can be generalized immediately to hyperbolic space $\H^3$ and the sphere $\s^3$, and Alexandrov's theorem to $\H^3$ and a hemisphere of $\s^3$.

An important problem from several viewpoints is to generalize the Hopf and Alexandrov theorems to more general ambient spaces - for instance, isoperimetric regions in a Riemannian $3$-manifold are bounded by compact embedded CMC surfaces. In this sense, among all possible choices of ambient spaces, the simply connected homogeneous $3$-manifolds are placed in a privileged position. Indeed, they are the most symmetric Riemannian $3$-manifolds other than the spaces of constant curvature, and are tightly linked to Thurston's $3$-dimensional geometries. Moreover, the global study of CMC surfaces in these homogeneous spaces in currently a topic of great activity.

Hopf's theorem was extended by Abresch and Rosenberg \cite{ar1,ar2} to all simply connected homogeneous $3$-manifolds with a $4$-dimensional isometry group, i.e., $\H^2\times\R$, $\s^2\times\R$, the Heisenberg group $\nil$, the universal cover of $\mathrm{PSL}_2(\R)$ and the Berger spheres: any immersed CMC sphere in any of these spaces is a standard rotational sphere. To do this, they proved the existence of a holomorphic quadratic differential, which is a linear combination of the Hopf differential and of a term coming from a certain ambient Killing field. Once there, the proof is similar to Hopf's: such a differential must vanish on a sphere, and this implies that the sphere is rotational. 

On the other hand, Alexandrov's theorem extends readily to $\H^2\times\R$ and a hemisphere of $\s^2$ times $\R$ in the following way: any compact embedded CMC surface is a standard rotational sphere (see for instance \cite{hsiang}). The key property of these ambient manifolds is that there exist reflections with respect to vertical planes, and this makes the Alexandrov reflection technique work. In contrast, the Alexandrov problem in $\nil$, the universal cover of $\mathrm{PSL}_2(\R)$ and the Berger hemispheres is still open, since there are no reflections in these manifolds.

The purpose of this paper is to investigate the Hopf problem in the simply connected homogeneous Lie group $\sol$, i.e., the only Thurston $3$-dimensional geometry where this problem remains open. The topic is a natural and widely commented extension of the Abresch-Rosenberg theorem, but in this $\sol$ setting there are substantial difficulties that do not appear in other homogeneous spaces.

One of these difficulties is that $\sol$ has an isometry group only of dimension $3$, and has no rotations. Hence, there are no known explicit CMC spheres, since, contrarily to other homogeneous $3$-manifolds, we cannot reduce the problem of finding CMC spheres to solving an ordinary differential equation (there are no compact one-parameter subgroups of ambient isometries). Let us also observe that geodesic spheres are not CMC \cite{lichnerowicz}. Moreover, even the existence of a CMC sphere for a specific mean curvature $H\in \R$ needs to be settled (although the existence of isoperimetric CMC spheres is known). Other basic difficulty is that the Abresch-Rosenberg quadratic differential does not exist in $\sol$ (more precisely, Abresch and Rosenberg claimed that there is no holomorphic quadratic differential of a certain form for CMC surfaces in $\sol$ \cite{ar2}).

As regards the Alexandrov problem in $\sol$, a key fact is that $\sol$ admits two foliations by totally geodesic surfaces such that reflections with respect to the leaves are isometries; this ensures that a compact embedded CMC surface is topologically a sphere (see \cite{egr}). Hence the problem of classifying compact embedded CMC surfaces is solved as soon as the Hopf problem is solved and embeddedness of the examples is studied. 

We now state the main theorems of this paper. We will generally assume without loss of generality that $H\geqslant 0$, by changing orientation if necessary. We also refer to Section \ref{sec:prelstab} for the basic definitions regarding stability, index and the Jacobi operator.


\begin{teo}\label{main}
Let $H>1/\sqrt3$. Then:
\begin{enumerate}
\item[$i)$]
There exists an embedded CMC $H$ sphere $S_H$ in $\sol$.
 \item[$ii)$]
Any immersed CMC $H$ sphere in $\sol$ differs from $S_H$ at most by a left translation.
 \item[$iii)$]
Any compact embedded CMC $H$ surface in $\sol$ differs from $S_H$ at most by a left translation.
\end{enumerate}
Moreover, these canonical spheres $S_H$ constitute a real analytic family, they all have index one and two reflection planes, and their Gauss maps are global diffeomorphisms into $\S^2$.
\end{teo}

As explained before, the Alexandrov-type uniqueness $iii)$ follows from the Hopf-type uniqueness $ii)$, by using the standard Alexandrov reflection technique with respect to the two canonical foliations of $\sol$ by totally geodesic surfaces. 

We actually have the following more general uniqueness theorem.

\begin{teo}\label{hopfmain}
Let $H>0$ such that there exists some immersed CMC $H$ sphere $\Sigma_H$ in $\sol$ verifying \emph{one} of the properties $(a)$-$(d)$, where actually $(a)\Rightarrow (b)\Rightarrow (c)\Rightarrow (d)$:
 \begin{enumerate}
 \item[$(a)$]
It is a solution to the isoperimetric problem in $\sol$.
 \item[$(b)$]
It is a (weakly) stable surface.
 \item[$(c)$]
It has index one.
 \item[$(d)$]
Its Gauss map is a (global) diffeomorphism into $\S^2$.
 \end{enumerate}
Then $\Sigma_H$ is embedded and unique up to left translations in $\sol$ among immersed CMC $H$ spheres (Hopf-type theorem) and among compact embedded CMC $H$ surfaces (Alexandrov-type theorem).
\end{teo}

Let us remark that solutions of the isoperimetric problem in $\sol$ are embedded CMC spheres. Hence, we can deduce from results of Pittet \cite{pittet} that the infimum of the set of $H>0$ such that there exists a CMC $H$ sphere satisfying $(a)$ is $0$ (see Section \ref{sec:alexandrov}). We will additionally prove that for all $H>1/\sqrt3$ there exists a CMC $H$ sphere satisfying $(c)$, which gives Theorem \ref{main}.

We outline the proof of the theorems. We first study the Gauss map of CMC $H$ immersions into $\sol$. We prove in Section \ref{sectiongauss} that the Gauss map is nowhere antiholomorphic and satisfies a certain second order elliptic equation. Conversely we obtain a Weierstrass-type representation formula that allows to recover a CMC $H$ immersion from a nowhere antiholomorphic solution of this elliptic equation.

To prove uniqueness of CMC $H$ spheres, the main idea will be to ensure the existence of a quadratic differential $Q\rmd z^2$ that satisfies the so-called \emph{Cauchy-Riemann inequality} (a property weaker than holomorphicity, introduced by Alencar, do Carmo and Tribuzy \cite{ACT}) for all CMC $H$ immersions. This will be the purpose of Section \ref{sec:uniqueness}. The main obstacle is that it seems very difficult and maybe impossible to obtain such a differential (or even just a CMC sphere) \emph{explicitly}. We are able to prove the existence of this differential provided there exists a CMC $H$ sphere whose Gauss map $G$ is a (global) diffeomorphism of $\s^2$ (our differential $Q\rmd z^2$ will be defined using this $G$).

The next step is to study the existence of CMC spheres whose Gauss map is a diffeomorphism. This is done in Section \ref{existence}. We first prove that the Gauss map of an isoperimetric sphere, and more generally of an index one CMC sphere, is a diffeomorphism. For this purpose we use a nodal domain argument. We also prove that a CMC sphere whose Gauss map is a diffeomorphism is embedded. 

Then we deform an isoperimetric sphere with large mean curvature by the implicit function theorem. More generally we prove that we can deform index one CMC spheres, and that the property of having index one is preserved by this deformation. In this way we prove that there exists an index one CMC $H$ sphere for all $H>1/\sqrt3$. To do this we need a bound on the second fundamental form and a bound on the diameter of the spheres. This diameter estimate is a consequence of a theorem of Rosenberg \cite{rosenbergaustralian} and relies on a stability argument; however, this estimate only holds for $H>1/\sqrt3$. This will complete the proof. 

We conjecture that Theorem \ref{main} should hold for every $H>0$ (see Section \ref{sec:conclusion}). Also, in this last section we will explain how our results give information on the symmetries of solutions to the isoperimetric problem in $\sol$.

A remarkable novelty of our approach to this problem is that we obtain a Hopf-type theorem for a class of surfaces without knowing explicitly beforehand the spheres for which uniqueness is aimed, or at least some key property of them (e.g. that they are rotational). This suggests that the ideas of our approach may be suitable for proving Hopf-type theorems in many other theories (see Remark \ref{comment}).

In this sense, let us mention that $\sol$ belongs to a $2$-parameter family of homogeneous $3$-manifolds, which also includes $\R^3$, $\H^3$ and $\H^2\times\R$. The manifolds in this family generically possess, as $\sol$, an isometry group of dimension $3$. However, they are ``less symmetric" than $\sol$, in the sense that their isometry group only has $4$ connected components, whereas the isometry group of $\sol$ has $8$ connected components. Also, contrarily to $\sol$, these manifolds do not have compact quotients (this is the reason why they are not Thurston geometries, see for instance \cite{bonahon}). We believe that it is very possible that the techniques developed in this paper also work in these manifolds, since we do not use these additional properties of $\sol$.

{\bf Acknowledgments.} This work was initiated during the workshop ``Research in Pairs: Surface Theory" at Kloster Sch\"ontal (Germany), March 2008. The authors would also like to thank Harold Rosenberg for many useful discussions on this subject.

\section{The Lie group $\sol$} \label{sol}

We will study surfaces immersed in the homogeneous Riemannian $3$-manifold $\sol$, i.e., the least symmetric of the eight canonical Thurston $3$-geometries. This preliminary section is intended to explain the basic geometric elements of this ambient space, and the consequences for CMC spheres of some known results.

\subsection{Isometries, connection and foliations}

The space $\sol$ can be viewed as $\R^3$ endowed with the Riemannian metric
$$\esiz, \esde = e^{2x_3}\rmd x_1^2+e^{-2x_3}\rmd x_2^2+\rmd x_3^2,$$ where $(x_1,x_2,x_3)$ are canonical coordinates of $\R^3$. It is important to observe that $\sol$ has a Lie group structure with respect to which the above metric is left-invariant. The group structure is given by the multiplication $$ (x_1,x_2,x_3)\cdot (y_1,y_2,y_3) =(x_1 +e^{-x_3}\,  y_1, x_2 + e^{x_3}\,  y_2, x_3 +y_3).$$ The isometry group of $\sol$ is easy to understand; it has dimension $3$ and the connected component of the identity is generated by the following three families of isometries:
$$(x_1,x_2,x_3)\mapsto(x_1+c,x_2,x_3),\quad
(x_1,x_2,x_3)\mapsto(x_1,x_2+c,x_3),$$
$$(x_1,x_2,x_3)\mapsto(e^{-c}x_1,e^cx_2,x_3+c).\quad$$ Obviously, these isometries are just \emph{left translations} in $\sol$ with respect to the Lie group structure above, i.e., left multiplications by elements in $\sol$. On the contrary, right translations are not isometries.

The corresponding Killing fields associated to these families of isometries are
$$F_1=\frac\partial{\partial x_1},\quad
F_2=\frac\partial{\partial x_2},\quad
F_3=-x_1\frac\partial{\partial x_1}+x_2\frac\partial{\partial x_2}+\frac\partial{\partial x_3}.$$ They are right-invariant. 


Another key property of $\sol$ is that it admits reflections. Indeed, Euclidean reflections in the $(x_1,x_2,x_3)$ coordinates with respect to the planes $x_1= {\rm const.}$ and $x_2 = {\rm const.}$ are orientation-reversing isometries of $\sol$. A very important consequence of this is that we can use the Alexandrov reflection technique in the $x_1$ and $x_2$ directions, as we will explain in Section \ref{sec:alexandrov}.

More specifically, 
the isotropy group of the origin $(0,0,0)$ is isomorphic to the dihedral group $\mathrm{D}_4$ and is generated by the following two isometries:
\begin{equation} \label{defsigmatau}
\sigma:(x_1,x_2,x_3)\mapsto(x_2,-x_1,-x_3),\quad
\tau:(x_1,x_2,x_3)\mapsto(-x_1,x_2,x_3).
\end{equation}
These two isometries are orientation-reversing; $\sigma$ has order $4$ and $\tau$ order $2$. Observe that the reflection with respect to the plane $x_2=0$ is given by $\sigma^2\tau$.

An important role in $\sol$ is played by the left-invariant orthonormal frame $(E_1,E_2,E_3)$ defined by
$$E_1=e^{-x_3}\frac\partial{\partial x_1},\quad
E_2=e^{x_3}\frac\partial{\partial x_2},\quad
E_3=\frac\partial{\partial x_3}.$$ We call it the \emph{canonical frame}. The coordinates with respect to the frame $(E_1,E_2,E_3)$ of a vector at a point $x=(x_1,x_2,x_3)\in\sol$ will be denoted into brackets; then we have
 \begin{equation}\label{compar}
 a_1\frac\partial{\partial x_1}+a_2\frac\partial{\partial x_2}+a_3\frac\partial{\partial x_3}=
\left[\begin{array}{c}
e^{x_3}a_1 \\
e^{-x_3}a_2 \\
a_3
      \end{array}\right].
 \end{equation}

The expression of the Riemannian connection $\hat\nabla$ of $\sol$ with respect to the canonical frame is the following:
\begin{equation} \label{connection}
\begin{array}{lll}
\hat\nabla_{E_1}E_1=-E_3, &
\hat\nabla_{E_2}E_1=0, &
\hat\nabla_{E_3}E_1=0, \\
\hat\nabla_{E_1}E_2=0, &
\hat\nabla_{E_2}E_2=E_3, &
\hat\nabla_{E_3}E_2=0, \\
\hat\nabla_{E_1}E_3=E_1, &
\hat\nabla_{E_2}E_3=-E_2, &
\hat\nabla_{E_3}E_3=0.
\end{array}
\end{equation}
From there, we see that the sectional curvatures of the planes $(E_2,E_3)$, $(E_1,E_3)$ and $(E_1,E_2)$ are $-1$, $-1$ and $1$ respectively, and that the Ricci curvature of $\sol$ is given, with respect to the canonical frame $(E_1,E_2,E_3)$, by $${\rm Ric} = \left(\begin{array}{ccc} 0 & 0 & 0 \\ 0 & 0 & 0 \\ 0 & 0 & -2
\end{array}\right).$$ In particular, $\sol$ has constant scalar curvature $-2$.

One of the nicest features of $\sol$ is the existence of three canonical foliations with good geometric properties. First, we have the foliations $$\cF_1 \equiv \{x_1 = {\rm constant} \}, \hspace{1cm} \cF_2 \equiv \{x_2 = {\rm constant} \} ,$$ which are orthogonal to the Killing fields $F_1,F_2$, respectively. 

It is immediate from the expression of the metric in $\sol$ that the leaves of this foliation are isometric to the hyperbolic plane $\H^2$. For instance, for a leaf $S$ of $\cF_1$, the coordinates $(x_2,x_3)\in \R^2$ are horospherical coordinates of $\H^2$, i.e. $$\esiz,\esde|_S= e^{-2x_3} \rmd x_2^2 + \rmd x_3^2, \hspace{1cm} (x_2,x_3)\in \R^2.$$ For us, the main property of these foliations is that Euclidean reflections across a leaf of $\cF_1,\cF_2$ are orientation-reversing isometries of $\sol$. Also, these leaves are the only totally geodesic surfaces in $\sol$ \cite{souamtoubiana}.

The third canonical foliation of $\sol$ is $\cG\equiv \{ x_3 = {\rm constant}\} $. Its leaves are isometric to $\R^2$ and are minimal.
This foliation $\cG$ is less important than $\cF_1,\cF_2$, since Euclidean reflections with respect to the planes $x_3={\rm constant}$ do not describe isometries in $\sol$ anymore. In addition, $\cG$ is no longer orthogonal to a Killing field of $\sol$.

The existence of these foliations by minimal surfaces also implies (by the maximum principle) that there is no compact minimal surface in $\sol$.

Let us also observe that the map $\sol\to\R,(x_1,x_2,x_3)\mapsto x_3$ is a Riemannian fibration. This means that the coordinate $x_3$ has a geometric meaning (whereas $x_1$ and $x_2$ do not have it).

For more details we refer to \cite{troyanov}. Some papers, such as \cite{troyanov} and \cite{inoguchi}, have a different convention for the metric of $\sol$.

\subsection{Alexandrov reflection and the isoperimetric problem} \label{sec:alexandrov}

Let us now focus on the Alexandrov reflection technique. Let $S$ be an embedded compact CMC surface in $\sol$. Then, applying the Alexandrov technique in the $x_1$ direction, we obtain that $S$ is, up to a left translation, a symmetric bigraph in the $x_1$ direction over some domain $U$ in the plane $x_1=0$. But then we can apply the Alexandrov reflection technique in the $x_2$ direction. This implies that, up to a left translation, $U$ can be written as
$$U=\{(x_2,x_3)\in\mathbb{R}^2;x_3\in I,x_2\in[-f(x_3),f(x_3)]\}$$ for some continuous function $f$ and some set $I$, which is necessarily an interval (otherwise $S$ would not be connected). Then $U$ is topologically a disk, and $S$ is topologically a sphere. Hence we have the following fundamental result (see the concluding remarks in \cite{egr}).

\begin{fact} \label{alexandrov}
A compact embedded CMC surface in $\sol$ is a sphere.
\end{fact}

In particular, all solutions of the isoperimetric problem in $\sol$ are spheres.

Some important information about the mean curvatures of the solutions of the isoperimetric problem in $\sol$ can be deduced from results of Pittet \cite{pittet}.

The identity component of the isometry group of $\sol$ is $\sol$ itself, acting by left multiplication. It is a solvable Lie group, hence amenable (an amenable Lie group is a compact extension of a solvable Lie group). It is unimodular: its Haar measure $\rmd x_1\rmd x_2\rmd x_3$ is biinvariant. It has exponential growth, i.e., the volume of a geodesic ball of radius $r$ increases exponentially with $r$ (see for example \cite{coulhon}). Consequently, by Theorem 2.1 in \cite{pittet}, there exist positive constants $c_1$ and $c_2$ such that the isoperimetric profile $I(v)$ of $\sol$ satisfies
\begin{equation} \label{profile}
\frac{c_1v}{\ln v}\leqslant I(v)\leqslant \frac{c_2v}{\ln v}
\end{equation}
for all $v$ large enough (let us recall that the isoperimetric profile $I(v)$ is defined as the infimum of the areas of compact surfaces enclosing a volume $v$).

It is well-known that $I$ admits left and right derivatives $I'_-(v)$ and $I'_+(v)$ at every $v\in(0,+\infty)$, and there exist isoperimetric surfaces of mean curvatures $I'_-(v)/2$ and $I'_+(v)/2$ respectively (see for instance \cite{rosclay}; here the mean curvature is computed for the unit normal pointing into the compact domain bounded by the surface, i.e., the mean curvature is positive because of the maximum principle used with respect to a minimal surface $x_3={\rm constant}$). Also by \eqref{profile} the numbers $I'_-(v)$ and $I'_+(v)$ cannot be bounded from below by a positive constant. From this and the fact that isoperimetric surfaces are spheres (since they are embedded), we get the following result.

\begin{fact} \label{infiso}
Let $\cJ$ be the set of real numbers $H>0$ such that there exists an isoperimetric sphere of mean curvature $H$. Then $$\inf\cJ=0.$$
\end{fact}

Even though we will not use it, it is worth pointing out the following consequence about \emph{entire graphs} (a surface is said to be an entire graph with respect to a non-zero Killing field $F$ of $\sol$ if it is transverse to $F$ and intersects every orbit of $F$ at exactly one point).
 
\begin{cor}
Any entire CMC graph with respect to a non-zero Killing field of $\sol$ must be minimal ($H=0$).
\end{cor}

\begin{proof}
Let $G$ be an entire CMC $H$ graph with respect to a non-zero Killing field $F$ with $H\neq 0$. By Fact \ref{infiso}, there exists an isoperimetric sphere $S$ whose mean curvature $H_0$ satisfies $H_0\leqslant|H|$.

Let $(\varphi_c)_{c\in\R}$ be the one-parameter group of isometries generated by $F$. Then $(\varphi_c(G))_{c\in\R}$ is a foliation of $\sol$. Let $c_0=\max\{c\in\R;\varphi_c(G)\cap S\neq\emptyset\}$ and $c_1=\min\{c\in\R;\varphi_c(G)\cap S\neq\emptyset\}$. Then at $c=c_0$ or at $c=c_1$ the sphere $S$ is tangent to $\varphi_c(G)$ and is situated in the mean convex side of $\varphi_c(G)$; this contradicts the maximum principle since $H_0\leqslant|H|$.
\end{proof}

The same proof also shows that no entire graph can have a mean curvature function bounded from below by a positive constant. In contrast, there exists various entire minimal graphs in $\sol$. Indeed, for instance one can easily check that a graph $x_1=f(x_2,x_3)$ is minimal if and only if
$$(e^{2x_3}f_{x_3}^2+1)f_{x_2x_2}-2e^{2x_3}f_{x_2}f_{x_3}f_{x_2x_3}
+(e^{-2x_3}+e^{2x_3}f_{x_2}^2)f_{x_3x_3}-(e^{2x_3}f_{x_2}^2-e^{-2x_3})f_{x_3}=0,$$
and so the following equations define entire $F_1$-graphs:
$$x_1=ax_2+b,\quad x_1=ae^{-x_3},\quad x_1=ax_2e^{-x_3},\quad x_1=x_2e^{-2x_3}.$$

\section{The Gauss map} \label{sectiongauss}

In this section we will expose the basic equations for immersed surfaces in $\sol$, putting special emphasis on the geometry of the Gauss map associated to the surface in terms of the Lie group structure of $\sol$.

So, let us consider an immersed oriented surface in $\sol$, that will be seen as a conformal immersion $X:\Sigma\to\sol$ of a Riemann surface $\Sigma$. We shall denote by $N:\Sigma\to\rmT\sol$ its unit normal.

If we fix a conformal coordinate $z=u+iv$ in $\Sigma$, then we have $$\langle X_z,X_{\bar z}\rangle=\frac{\lambda}2>0, \hspace{1cm} \langle X_z,X_z\rangle=0,$$ where $\landa$ is the conformal factor of the metric with respecto to $z$. Moreover, we will denote the coordinates of $X_z$ and $N$ with respect to the canonical frame $(E_1,E_2,E_3)$ by
$$X_z=\left[\begin{array}{c} A_1 \\ A_2 \\ A_3 \end{array}\right],\quad
N=\left[\begin{array}{c} N_1 \\ N_2 \\ N_3 \end{array}\right].$$
The usual \emph{Hopf differential} of $X$, i.e., the $(2,0)$ part of its complexified second fundamental form, is defined as $$P\rmd z^2=\langle N,\hat\nabla_{X_z}{X_z}\rangle\rmd z^2.$$
From the definitions, we have the basic algebraic relations
\begin{equation} \label{sumAk}
\left\{\def\arraystretch{1.4}\begin{array}{l}
|A_1|^2+|A_2|^2+|A_3|^2=\displaystyle \frac{\lambda}2,\quad \\
A_1^2+A_2^2+A_3^2=0,\quad \\
N_1^2+N_2^2+N_3^2=1,\quad \\
A_1N_1+A_2N_2+A_3N_3=0.
\end{array}\right.
\end{equation}
A classical computation proves that the Gauss-Weingarten equations of the immersion read as 
\begin{equation} \label{diffX}
\left\{\def\arraystretch{1.9}\begin{array}{l}
\hat\nabla_{X_z}X_z=\displaystyle \frac{\lambda_z}{\lambda}X_z+PN,\quad \\
\hat\nabla_{X_{\bar z}}X_z= \displaystyle \frac{\lambda H}2N,\quad \\
\hat\nabla_{X_{\bar z}}X_{\bar z}=\displaystyle\frac{\lambda_{\bar z}}{\lambda}X_{\bar z}+\bar PN, \\
 \end{array}\right. \hspace{1cm} 
\left\{\def\arraystretch{1.9}\begin{array}{l}
\hat\nabla_{X_z}N=-HX_z-\displaystyle\frac{2P}{\lambda}X_{\bar z},\quad \\
\hat\nabla_{X_{\bar z}}N=-\displaystyle\frac{2\bar P}{\lambda}X_z-HX_{\bar z}. \\
 \end{array}\right.
\end{equation}
Using \eqref{connection} in these equations we get
\begin{equation} \label{diffAkz}
\left\{\def\arraystretch{1.9} \begin{array}{l}
\displaystyle{A_{1z}=\frac{\lambda_z}{\lambda}A_1+PN_1-A_1A_3,} \\
\displaystyle{A_{2z}=\frac{\lambda_z}{\lambda}A_2+PN_2+A_2A_3,} \\
\displaystyle{A_{3z}=\frac{\lambda_z}{\lambda}A_3+PN_3+A_1^2-A_2^2,}
\end{array}\right.
\quad
\left\{\def\arraystretch{1.9} \begin{array}{l}
\displaystyle{A_{1\bar z}=\frac{\lambda H}2N_1-\bar A_1A_3,} \\
\displaystyle{A_{2\bar z}=\frac{\lambda H}2N_2+\bar A_2A_3,} \\
\displaystyle{A_{3\bar z}=\frac{\lambda H}2N_3+|A_1|^2-|A_2|^2,}
\end{array}\right.
\end{equation}
\begin{equation} \label{diffNkz}
\left\{\def\arraystretch{1.9}\begin{array}{l}
\displaystyle{N_{1z}=-HA_1-\frac{2P}{\lambda}\bar A_1-A_1N_3,} \\
\displaystyle{N_{2z}=-HA_2-\frac{2P}{\lambda}\bar A_2+A_2N_3,} \\
\displaystyle{N_{3z}=-HA_3-\frac{2P}{\lambda}\bar A_3+A_1N_1-A_2N_2.}
\end{array}\right.
\end{equation}
Moreover, the fact that $X_z\times X_{\bar z}=i\frac{\lambda}2N$ implies that
\begin{equation} \label{relationNkAk}
N_1=-\frac{2i}{\lambda}(A_2\bar A_3-A_3\bar A_2),\quad
N_2=-\frac{2i}{\lambda}(A_3\bar A_1-A_1\bar A_3),\quad
N_3=-\frac{2i}{\lambda}(A_1\bar A_2-A_2\bar A_1).
\end{equation}

Once here, let us define the \emph{Gauss map} associated to the surface. We first set 
$$\hat N=(N_1,N_2,N_3):\Sigma\flecha\s^2\subset \R^3.$$ In other words, for each $z\in \Sigma$, $\hat{N}(z)$ is just the vector in the Lie algebra of $\sol$ (identified with $\R^3$ by means of the canonical frame) that corresponds to $N(z)$ when we apply a left translation in $\sol$ taking $X(z)$ to the origin.

\begin{defi}
Given an immersed oriented surface $X:\Sigma\flecha \sol$, the \emph{Gauss map} of $X$ is the map $$g:=\varphi\circ\hat N:\Sigma\to\bar\C:=\C\cup \{\8\},$$ where $\varphi$ is the stereographic projection with respect to the southern pole, i.e.
$$g=\frac{N_1+iN_2}{1+N_3}.$$ Here, $N_1,N_2,N_3$ are the coordinates of the unit normal of $X$ with respect to the canonical frame $(E_1,E_2,E_3)$.
\end{defi}

Equivalently, we have 
\begin{equation}\label{tres55}
N=\frac1{1+|g|^2}\left[\begin{array}{c} 2\re g \\ 2\im g \\ 1-|g|^2
          \end{array}\right].\end{equation}

\begin{remark}\label{rem1}
The Gauss map obviously remains invariant when we apply a left translation to the surface. On the other hand, if we apply the orientation-preserving isometry of $\sol$ $$\sigma\tau:(x_1,x_2,x_3)\mapsto(x_2,x_1,-x_3)$$ then the Gauss map $g$ of the surface changes to $g^* := i/g$.
\end{remark}

\begin{defi}
Given an immersed oriented surface $X:\Sigma\flecha \sol$, let us denote $$\eta:=2\langle E_3,X_z\rangle = 2 A_3.$$ 
\end{defi}

A straightforward computation from \eqref{sumAk} proves that 
 \begin{equation}\label{asubis}
 A_1=-\frac{1-\bar g^2}{4\bar g}\eta,\quad
A_2=i\frac{1+\bar g^2}{4\bar g}\eta,\quad A_3=\frac{\eta}2
 \end{equation}
and thereby 
 \begin{equation}\label{landa}
 \lambda=(1+|g|^2)^2\frac{|\eta|^2}{4|g|^2}.
  \end{equation}

Then the right system in \eqref{diffAkz} becomes
\begin{equation}\label{1line}(\bar g^2+1)\frac{\eta\bar g_{\bar z}}{4\bar g^2}+\frac14\left(\bar g-\frac1{\bar g}\right)\eta_{\bar z}
=(1+|g|^2)\frac{|\eta|^2}{4|g|^2}H\re(g)+|\eta|^2\frac{1-g^2}{8g},\end{equation}
\begin{equation}\label{2line}i(\bar g^2-1)\frac{\eta\bar g_{\bar z}}{4\bar g^2}+\frac i4\left(\bar g+\frac1{\bar g}\right)\eta_{\bar z}
=(1+|g|^2)\frac{|\eta|^2}{4|g|^2}H\im(g)-i|\eta|^2\frac{1+g^2}{8g},\end{equation}
\begin{equation}\label{3rdline}
\frac12\eta_{\bar z}=\frac{|\eta|^2}{8|g|^2}((1-|g|^4)H-g^2-\bar g^2).\end{equation}
Reporting \eqref{3rdline} into \eqref{1line} + $i$ \eqref{2line} gives
$$\frac{g\bar g_{\bar z}}{2\bar\eta}=\frac H8(1+|g|^2)^2+\frac18(\bar g^2-g^2),$$
i.e.
\begin{equation} \label{formulaeta}
\eta=\frac{4\bar gg_z}{R(g)} \hspace{1cm} \text{ where } \hspace{0.4cm}  R(q)=H(1+|q|^2)^2+q^2-\bar q^2.
\end{equation}
Once here, we are ready to state the main result of this section.

\begin{teo}\label{rep}
Let $X:\Sigma\flecha \sol$ be a CMC $H$ surface with Gauss map $g:\Sigma\flecha \bar{\C}$. Then, $g$ is nowhere antiholomorphic, i.e., $g_z\neq 0$ at every point for any local conformal parameter $z$ on $\Sigma$, and $g$ verifies the second order elliptic equation 
\begin{equation} \label{eqg2}
g_{z\bar z}=A(g)g_zg_{\bar z}+B(g)g_z\bar g_{\bar z},
\end{equation}
where, by definition, 
\begin{equation} \label{defAB2}
A(q)=\frac{R_q}R=\frac{2H(1+|q|^2)\bar q+2q}{R(q)},\quad
B(q)=\frac{R_{\bar q}}R-\frac{\bar R_{\bar q}}{\bar R}=-\frac{4H(1+|q|^2)(\bar q+q^3)}{|R(q)|^2},
\end{equation}
$$R(q)=H(1+|q|^2)^2+q^2-\bar q^2.$$

Moreover, the immersion $X=(x_1,x_2,x_3):\Sigma\to\sol$ can be recovered in terms of the Gauss map $g$ by means of the representation formula 
\begin{equation}\label{repfor}
(x_1)_z=e^{-x_3}\frac{(\bar{g}^2 -1)g_z}{R(g)},
\quad(x_2)_z=ie^{x_3}\frac{(\bar{g}^2 +1)g_z}{R(g)},
\quad(x_3)_z=\frac{2\bar{g}g_z}{R(g)}.
 \end{equation}

Conversely, if a map $g:\Sigma\flecha \bar{\C}$ from a simply connected Riemann surface $\Sigma$ verifies \eqref{eqg2} for $H\neq 0$ and the coefficients $A,B$ given in \eqref{defAB2}, and if $g$ is nowhere antiholomorphic, then the map $X:\Sigma\flecha \sol$ given by the representation formula \eqref{repfor} defines a CMC $H$ surface in $\sol$ whose Gauss map is $g$.
\end{teo}


\begin{remark}\label{remg}
It is immediate that \eqref{eqg2} is invariant by changes of conformal parameter, i.e., $g$ is a solution to \eqref{eqg2} if and only if $g\circ \psi$ is a solution to \eqref{eqg2}, where $\psi$ is a locally injective meromorphic function. This shows that we can work with \eqref{eqg2} at $z=\8$, just by making the change $z\mapsto 1/z$.

In addition, a direct computation shows that if $g$ is a solution to \eqref{eqg2}, then $$g^*:= \frac{i}{g}$$ is also a solution to \eqref{eqg2} (see Remark \ref{rem1} for the geometric meaning of this duality). Again, this gives a way to work with \eqref{eqg2} at points where $g=\8$.
\end{remark}

\begin{remark}
If $g$ is a nowhere antiholomorphic solution to \eqref{eqg2}, inducing a CMC $H$ immersion $X$, then a direct computation shows that $ig$ and $1/g$ are also solutions to \eqref{eqg2} with $H$ replaced by $-H$, and they induce the CMC $-H$ immersion $\sigma\circ X$ and $\tau\circ X$. Also, the map $z\mapsto 1/\bar g(\bar z)$ is also solution to \eqref{eqg2} with $H$ replaced by $-H$, and induces the same surface as $g$ but with the opposite orientation. 
\end{remark}

\begin{proof}
The fact that $g$ is nowhere antiholomorphic follows from \eqref{landa} and \eqref{formulaeta}.

Formula \eqref{eqg2} follows from the constancy of $H$, just by computing the expression $\eta_{\bar z}/\eta$ from \eqref{formulaeta} and reporting it then in \eqref{3rdline}. Besides, the representation formula \eqref{repfor} follows directly from \eqref{asubis} and \eqref{formulaeta}, taking into account the relation \eqref{compar}.

To prove the converse, we start with a simply connected domain $\cU\subset \Sigma$ on which $g\neq \8$ at every point. Then, using \eqref{eqg2}, we have $$\frac{\parc}{\parc \bar{z}} \left( \frac{2 \bar{g} g_z}{R(g)}\right) = \frac{2 |g_z|^2}{|R(g)|^2} \left( H (1-|g|^4) - (g^2 +\bar{g}^2)\right) \in \R.$$ Hence, there exists $x_3:\cU\subset \Sigma\flecha \R$ with 
 \begin{equation}\label{forx3}
 (x_3)_z = \frac{2\bar{g} g_z}{R(g)}.
 \end{equation}
We define now in terms of $x_3$ the map $\cF:\cU\flecha \C^3$ given by 
\begin{equation}\label{laF}
\cF =\left(e^{-x_3} \frac{(\bar{g}^2 -1) g_z}{R(g)}, i\, e^{x_3} \frac{(\bar{g}^2 +1) g_z}{R(g)}, \frac{2\bar{g} g_z}{R(g)}\right).
\end{equation} 
Again using \eqref{eqg2} it can be checked that $\cF_{\bar{z}}\in \R$ at every point. So, there exists $X=(x_1,x_2,x_3):\cU\flecha \R^3\equiv \sol$ with $X_z =\cF$. This indicates that the representation formula \eqref{repfor} gives, indeed, a well defined map $X:\cU\subset \Sigma\flecha \sol$. Clearly, $X$ is defined up to left translation in $\sol$, due to the integration constants involved in \eqref{repfor}.

Now, let us assume that $\cU\subset \Sigma$ is a simply connected domain on which $g$ may take the value $\8$ at some points, but not the value $0$. Let us define $g^* =i/g$. By Remark \ref{remg}, $g^*$ is a solution to \eqref{eqg2} not taking the value $\8$ on $\cU$, and so we already have proved the existence of a map $X^*=(x_1^*,x_2^*,x_3^*):\cU\flecha \sol$ with $$X_z^* = \cF^* = \left(e^{-x_3} \frac{(\bar{g^*}^2 -1) g^*_z}{R(g^*)}, i\, e^{x_3} \frac{(\bar{g^*}^2 +1) g^*_z}{R(g^*)}, \frac{2\bar{g^*} g^*_z}{R(g^*)}\right).$$ Hence, since $g^* =i/g$, a direct computation shows that $X:= (x_2^*, x_1^*, -x_3^*)$ satisfies $X_z=\cF$ for $\cF$ as in \eqref{laF}. Integrating $\cF$ again shows that the representation formula \eqref{repfor} is well defined also when $g=\8$ at some points. Finally, as $\Sigma$ is simply connected, \eqref{repfor} can be defined globally on $\Sigma$, thus giving rise to a map $X:\Sigma\flecha \sol$ which is unique up to left translations.

Once here, we need to check that $X$ is a conformal CMC $H$ immersion with Gauss map $g$. In this sense, it is clear from \eqref{laF} and \eqref{compar} that $$X_z= [A_1,A_2,A_3],$$ where the $A_k$'s are given by \eqref{asubis} and \eqref{formulaeta}. Consequently, $\esiz X_z,X_z\esde=0$ and 
\begin{equation}\label{landag}
\esiz X_z,X_{\bar{z}}\esde = \frac{\landa}{2} = \frac{2(1+|g|^2)^2}{|R(g)|^2} \, |g_z|^2,
\end{equation}
which is non-zero since $g$ is nowhere antiholomorphic. Thus, $X$ is a conformal immersion.

Besides, denoting by $N=(-2i/\landa) X_z\times X_{\bar{z}}$ the unit normal of $X$, we easily see from \eqref{relationNkAk} that \eqref{tres55} holds. So, $g$ is indeed the Gauss map of the surface $X$.

At last, let $H_X$ denote the mean curvature function of $X$. Putting together \eqref{asubis} and \eqref{formulaeta} we see that $H_X$ is given (for any immersed surface in $\sol$) by $$H_X=\frac{2\bar{g} g_z}{(1+|g|^2)^2 A_3} - \frac{g^2 -\bar{g}^2}{(1+|g|^2)^2}.$$ 
In our case, as \eqref{asubis} and \eqref{formulaeta} hold for $H$, this tells directly that $H=H_X$, as wished. This completes the proof.
\end{proof}

\begin{remark} \label{formulaP}
Using \eqref{diffNkz}, \eqref{diffAkz}, \eqref{relationNkAk}, \eqref{asubis} and \eqref{formulaeta}, we obtain after a computation that the Hopf differential $P\, \rmd z^2$ of the surface is given by 
\begin{eqnarray*}
P & = & -A_1N_{1z}-A_2N_{2z}-A_3N_{3z}+N_3(A_2^2-A_1^2)+A_3(A_1N_1-A_2N_2) \\
& = & \frac{2g_z\bar g_z}{R(g)}-\frac{2(1-\bar g^4)g_z^2}{R(g)^2}.
\end{eqnarray*}
\end{remark}

Let us now make some brief comments regarding the special case of minimal surfaces in $\sol$, i.e. the case $H=0$. In that situation it is immediate from \eqref{defAB2} that $B(g)=0$, and thus the Gauss map equation \eqref{eqg2} simplifies to 
 \begin{equation}\label{gmin}
 g_{z\bar{z}} = \frac{R^*_g}{R^*} \, g_z g_{\bar{z}} ,\hspace{1cm} R^*(q):= q^2 -\bar{q}^2.
 \end{equation}
Now, it is immediate that this is the harmonic map equation for maps $g:\Sigma\flecha (\bar{\C}, \rmd\sigma^2),$ where $\rmd\sigma^2$ is the singular metric on $\bar\C$ given by $$\rmd\sigma^2 = \frac{|\rmd w|^2}{|w^2-\bar{w}^2|}.$$ This result was first obtained in \cite{inoguchi}. It is somehow parallel to the case of minimal surfaces in the Heisenberg space ${\rm Nil}_3$. Indeed, it is a result of the first author \cite{Dan2}  that the Gauss map of a minimal local graph in $\nil$ (with respect to the canonical Riemannian fibration $\nil\to\R^2$) is a harmonic map into the unit disk $\D$ endowed with the Poincar\'e metric.

In addition, it is remarkable that \emph{every minimal surface in $\sol$ has an associated holomorphic quadratic differential}. Indeed, it follows from \eqref{gmin} that the quadratic differential $$Q^* \rmd z^2 := \frac{1}{R^* (g)} g_z \bar{g}_z \, \rmd z^2 $$ is holomorphic (whenever it is well defined) on any minimal surface in $\sol$.

\begin{remark}
Theorem \ref{rep} has been obtained by analyzing the structure equations of immersed CMC surfaces in $\sol$. Alternatively, it could also have been obtained by working with the Dirac equations obtained by Berdinsky and Taimanov in \cite{taimanov}.
\end{remark}

\begin{eje}
Let us compute the CMC surfaces in $\sol$ that are invariant by translations in the $x_2$ direction. Then we have to look for a real-valued Gauss map $g$, and we may choose without loss of generality a conformal parameter $z=u+iv$ such that $g$ only depends on $u$. Then \eqref{eqg2} becomes
$$g_{uu}=(A(g)+B(g))g_u^2=\frac{2Hg(1+g^2)-2g}{H(1+g^2)^2}g_u^2.$$ From this we get
\begin{equation} \label{eqginvx2}
g_u=\beta(1+g^2)\exp\left(\frac1{H(1+g^2)}\right)
\end{equation}
for some constant $\beta\in\R\setminus\{0\}$. Up to a multiplication of the conformal parameter by a constant, we may assume that $\beta=1$. 

Conversely, assume that $g$ satisfies $g_v=0$ and \eqref{eqginvx2} with $\beta=1$. Let us first notice that such a real-valued function is defined on a domain $(u_0,u_1)\times\R$ where $u_0$ and $u_1$ are finite, since $$\int \frac{1}{1+x^2} \exp\left(-{\frac1{H(1+x^2)}}\right)\rmd x$$ is bounded. On the other hand, the function $\tilde g:=1/g$ satisfies
$$\tilde g_u=-(1+\tilde g^2)\exp\left(\frac{\tilde g^2}{H(1+\tilde g^2)}\right).$$ Hence by allowing the value $\infty$ we can define a function $g:\C\to\R\cup\{\infty\}$, which will be periodic in $u$.
Then, \eqref{formulaeta} gives $$A_3=\frac{gg_u}{H(1+g^2)^2},$$ and so, up to a translation, $$x_3=-\frac1{H(1+g^2)}+\frac1{2H}.$$ Then we see from \eqref{asubis} and \eqref{eqginvx2} that
$$A_1=\frac{(g^2-1)g_u}{2H(1+g^2)^2}=\frac{g^2-1}{2H(1+g^2)}e^{-x_3+1/(2H)},
\quad A_2=\frac{ig_u}{2H(1+g^2)}=\frac i{2H}e^{-x_3+1/(2H)},$$ so, up to a translation in the $x_2$ direction, we get from \eqref{compar}
$$x_1=\int e^{-2x_3+1/(2H)}\frac{g^2-1}{H(1+g^2)}\rmd u,\quad x_2=-\frac{e^{1/(2H)}}Hv.$$

We now define a new variable $t$ by
$$\cos t=\frac{g^2-1}{1+g^2},\quad\sin t=\frac{2g}{1+g^2}.$$ Then from \eqref{eqginvx2} we have $$t_u =-2\exp\left(\frac{1}{H(1+g^2)}\right),$$ and hence we get that the immersion $X=(x_1,x_2,x_3)$ is given with respect to these $(t,v)$ coordinates by
$$x_1(t,v)=-\frac1{2H}\int e^{-(\cos t)/(2H)}\cos t\,\rmd t,
\quad x_2(t,v)=-\frac{e^{1/(2H)}}Hv,
\quad x_3(t,v)=\frac{\cos t}{2H}.$$
This surface is complete, simply connected and not embedded since
$$\int_{-\pi/2}^{3\pi/2}e^{-(\cos t)/(2H)}\cos t\,\rmd t
=\int_{-\pi/2}^{\pi/2}-2\sinh\left(\frac{\cos t}{2H}\right)\cos t\,\rmd t<0.$$ 
Also, it is conformally equivalent to the complex plane $\C$. The profile curve is drawn in Figure \ref{cyl1}.

\begin{figure}[htbp]
\begin{center}
\includegraphics[width=0.45\textwidth]{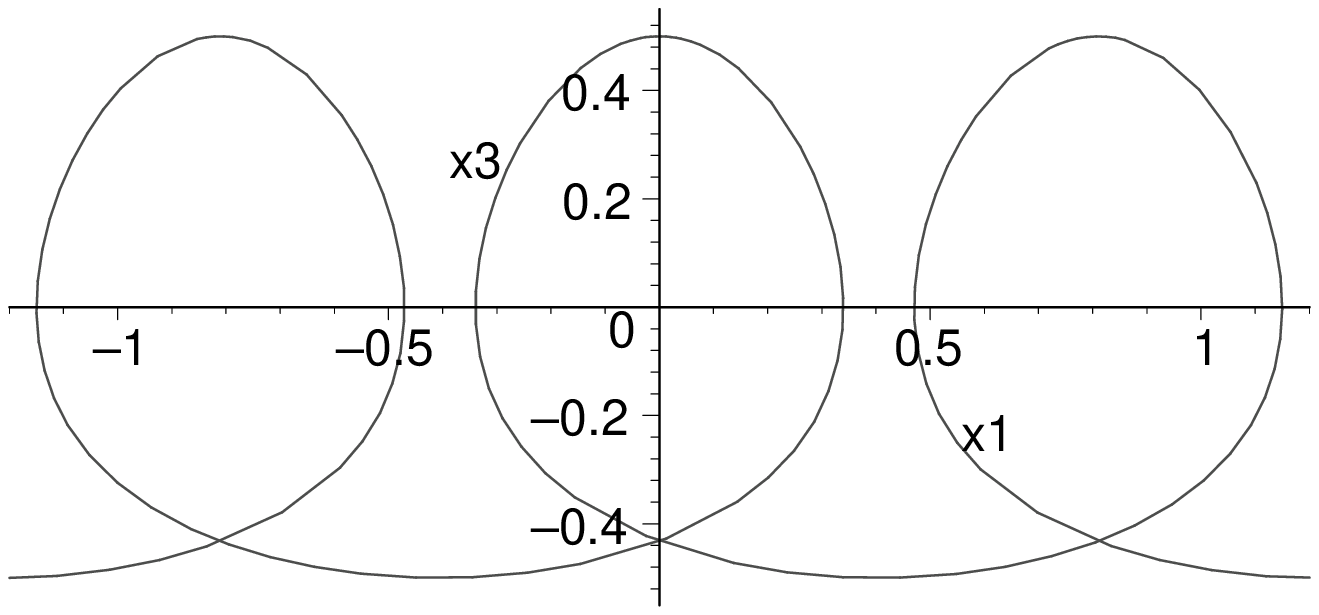} \quad\quad
\includegraphics[width=0.45\textwidth]{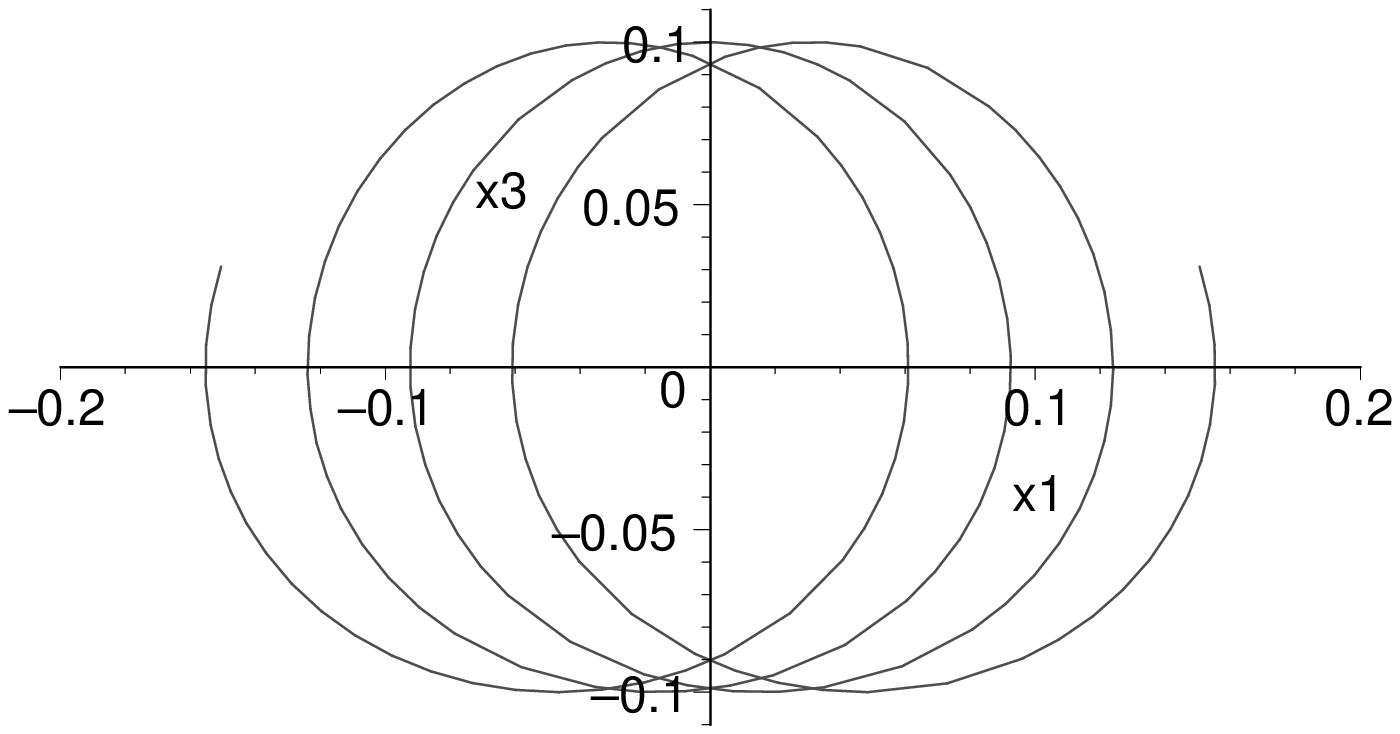}
\caption{The profile curve in the $(x_1,x_3)$-plane when $H=1$ (left) and $H=5$ (right).}
\label{cyl1}
\end{center}
\end{figure}


\end{eje}

\section{Uniqueness of immersed CMC spheres} \label{sec:uniqueness}

The uniqueness problem for a class of immersed spheres is usually approached by seeking a holomorphic quadratic differential $Q\, \rmd z^2$ for the class of surfaces under study. Once this is done, this holomorphic differential will vanish on spheres, what provides the key step for proving uniqueness.

As a matter of fact, and as already pointed out by Hopf, is not necessary that $Q \, \rmd z^2$ be holomorphic. Indeed, it suffices to find a complex quadratic differential $Q\, \rmd z^2$ whose zeros (when not identically zero) are isolated and of negative index. This condition implies again, by the Poincar\'e-Hopf theorem, that $Q \, \rmd z^2$ will vanish on topological spheres.

By means of this idea and a careful local analysis, Alencar, do Carmo and Tribuzy recently obtained in \cite{ACT} the following result:

\begin{teo}\label{crt}
Let $Q \, \rmd z^2$ denote a complex quadratic differential on $\bar{\C}$. Assume that around every point $z_0$ of $\bar{\C}$ we have 
 \begin{equation}\label{cari}
 |Q_{\bar{z}}|^2 \leqslant f \, |Q|^2,
  \end{equation} 
where $f$ is a continuous non-negative real function around $z_0$, and $z$ is a local conformal parameter. Then $Q\equiv 0$ on $\bar{\C}$.
\end{teo}

Our next objective is to seek a quadratic differential $Q \, \rmd z^2$ satisfying the so-called \emph{Cauchy-Riemann inequality} \eqref{cari} for every CMC $H$ surface in $\sol$. Our main result in this section will be that such a quadratic differential exists, provided there also exists an immersed CMC $H$ sphere in $\sol$ whose Gauss map is a global diffeomorphism of $\bar{\C}\equiv \S^2$. The existence of such a sphere will be discussed in Section \ref{existence}.

We will make a constructive approach to this result, in order to clarify its nature. To start, we consider a very general quadratic differential $Q \, \rmd z^2$ on a CMC $H$ surface $X:\Sigma\flecha \sol$, defined in terms of its Gauss map $g:\Sigma\flecha \bar{\C}$ as 
\begin{equation}\label{qfor}
Q = L(g) g_z^2 + M(g) g_z \bar{g}_z,
\end{equation}
where $$\begin{array}{ll}
L: & \bar\C\to\C \\
& q \mapsto L(q),\end{array}\quad\quad
\begin{array}{ll}
M: & \bar\C\to\C \\
& q \mapsto M(q)\end{array}$$
are to be determined. It is immediate that $Q\, \rmd z^2$ is invariant by conformal changes of coordinates in $\Sigma$, so it gives a well defined quadratic differential, at least at points where $g\neq \8$. In order to ensure that $Q$ is well defined at points where $g=\8$, we use the conformal chart $w =i/q$ on $\bar{\C}$, the Riemann sphere on which $g$ takes its values. From there, we get the following restrictions, which imply that $Q$ is well defined at every point:
\begin{equation}\label{condl}
q^4 L(q) \text{ is smooth and bounded around $q=\8$,}
\end{equation}
and 
\begin{equation}\label{condm}
|q|^4 M(q) \text{ is smooth and bounded around $q=\8$.}
\end{equation}

Then, using \eqref{eqg2} we get from \eqref{qfor}
\begin{equation}\label{qz0}
\def\arraystretch{1.5}\begin{array}{lll}
Q_{\bar{z}} &= & g_{\bar{z}} g_z^2 (L_q +2LA) + g_z |g_z|^2 (L_{\bar{q}} + 2LB + M\bar{B}) \\ & & + g_z |g_{\bar{z}}|^2 (M_q + MA) + \bar{g}_z |g_z|^2 ( M_{\bar{q}} + MB + M\bar{A}).
\end{array}
\end{equation}

We first notice that this expression simplifies for some specific choice for $M$. Indeed, let us define $M:\C\flecha \C$ by 
\begin{equation} \label{defM}
M(q)=\frac1{R(q)}=\frac1{H(1+|q|^2)^2+q^2-\bar q^2}.
\end{equation}
A computation shows that 
 \begin{equation}\label{invm}
 M\left(\frac{i}{q}\right)= |q|^4 M(q),
  \end{equation}
what implies that $M$ can be extended to a map $M:\bar{\C}\flecha \C$ satisfying \eqref{condm} by setting
$$M(\infty)=0.$$
Moreover, the following formulas are a direct consequence of the definition of the coefficients $A,B$:
\begin{equation} \label{diffM}
M_q=-MA,\quad M_{\bar q}=-M(\bar A+B).
\end{equation}
Thereby, if in the expression \eqref{qfor} we choose $M$ as in \eqref{defM}, then \eqref{qz0} simplifies to
\begin{equation}\label{qz}
Q_{\bar{z}}= g_z^2 \left\{ (L_q + 2LA) g_{\bar{z}} + (L_{\bar{q}} + 2LB + M\bar{B}) \bar g_{\bar z} \right\}.
\end{equation}

From here, we have the following result.

\begin{lem}\label{lemq}
Let $L:\bar{\C}\flecha \C$ be a global solution of the differential equation 
\begin{equation} \label{eqL}
(L_q+2LA)\bar L=(L_{\bar q}+2LB+M\bar B)\bar M,
\end{equation}
where $A,B,M$ are given by \eqref{defAB2} and \eqref{defM}, and which satisfies the condition \eqref{condl}. Then, for any immersed CMC $H$ surface $X:\Sigma\flecha \sol$ with Gauss map $g:\Sigma\flecha \bar{\C}$, the quadratic differential $Q(g)\rmd z^2$ on $\Sigma$ given by $$Q(g)= L(g) g_z^2 + M(g) g_z \bar{g}_z $$ satisfies the Cauchy-Riemann inequality \eqref{cari} at every point of $\Sigma$.
\end{lem}

\begin{proof}
Let $g:\Sigma\flecha \bar{\C}$ denote the Gauss map of a CMC $H$ surface in $\sol$, and define 
\begin{equation}\label{defalfa}
\alpha=\frac{L_q+2LA}{\bar M}.
\end{equation}
This function is well defined on $\C$, since $M$ does not vanish on $\C$. We need to study $\alpha(q)$ in a neighbourhood of $q=\infty$. For this we shall use that, since $L$ verifies the growth condition \eqref{condl}, we have \begin{equation} \label{Linfinity}
L(q)=\frac1{q^4}\, \varphi\left(\frac iq\right)
\end{equation}
where $\varphi$ is a smooth bounded function in a neighbourhood of $0$. Using \eqref{Linfinity} together with
\eqref{invm} and the easily verified relation 
$$A\left(\frac{i}{q}\right)= i q^2 A(q) - 2 i q,$$
we get from \eqref{defalfa} that
\begin{equation} \label{alphainfinity}
\alpha(q)=-\frac{i\bar q^2}{q^4}\bar R\left(\frac iq\right)\left(\varphi_q\left(\frac iq\right)
+2\varphi\left(\frac iq\right)A\left(\frac iq\right)\right)
=\frac{\bar q^2}{q^4}\psi\left(\frac iq\right)
\end{equation}
where $\psi$ is bounded and continuous in a neighbourhood of $0$.

Once here, observe that by \eqref{qz} and \eqref{eqL} we have $$ Q_{\bar{z}} = \alfa(g)\overline{(L(g) g_z + M(g) \bar{g}_z)} \, g_z^2,$$ and so $$|Q_{\bar{z}}| = \beta \, |Q|, \hspace{0.5cm} \text{ with } \hspace{0.5cm} \beta(z)= |\alfa (g(z)) g_z(z)|.$$ This function $\beta$ is continuous except, possibly, at points $z\in \Sigma$ where $g(z)=\infty$. But in the neighbourhood of such a point, by \eqref{alphainfinity} we have
$$\beta=\left|\psi\left(\frac ig\right)\frac{\partial}{\partial z}\left(\frac ig\right)\right|.$$ Consequently $\beta$ is continuous at every $z\in \Sigma$. Thus, $Q$ verifies the Cauchy-Riemann inequality \eqref{cari}, as claimed.
\end{proof}

Once here, it comes clear from Theorem \ref{crt} and Lemma \ref{lemq} that, in order to have a quadratic differential that vanishes on CMC $H$ spheres in $\sol$, we just need a global solution $L$ to \eqref{eqL} that satisfies the growth condition \eqref{condl}. Next, we prove that such a solution exists \emph{provided} we know beforehand the existence of a CMC $H$ sphere whose Gauss map is a global diffeomorphism of $\bar{\C}$. 

\begin{pro} \label{existenceQ}
Let $H\neq0$. Assume that there exists a CMC $H$ sphere $S_H\equiv\bar\C$ whose Gauss map $G:\bar\C\to\bar\C$ is a diffeomorphism. Then there exists a solution $L:\bar\C\to\C$ to \eqref{eqL} that satisfies the condition at infinity \eqref{condl}. 

Consequently, we can define, associated to every smooth map $g:\Sigma\to\bar\C$ from a Riemann surface $\Sigma$, the quadratic differential $Q(g)\rmd z^2$ where
$$Q(g)=L(g)g_z^2+M(g)g_z\bar g_z,$$ and then $Q(g)\rmd z^2$ satisfies the Cauchy-Riemann inequality \eqref{cari} if $g:\Sigma\to\bar{\C}$ is the Gauss map of any CMC $H$ immersion $X:\Sigma\to\sol$ from a Riemann surface $\Sigma$.
\end{pro}


\begin{proof}
We view $S_H$ as a conformal immersion from $\bar{\C}$ into $\sol$, whose Gauss map $G:\bar{\C}\flecha \bar{\C}$ is a global diffeomorphism, and $G_z\neq 0$ at every point. Then, we can define a function $L:\C\to\C$ by
\begin{equation}\label{elege}
L(G(z))=-\frac{M(G(z))\bar G_z(z)}{G_z(z)}.
\end{equation}
It is important to remark that this function $L$ has been chosen so that 
\begin{equation}\label{qzg}
L(G) G_z^2 + M(G) G_z \bar{G}_z =0
\end{equation}
identically on $\bar{\C}$ (we prove next that $L$ is well defined at $\8$).

Let us analyze the behaviour of $L(q)$ at $q=\infty$. We set $G=i/\Gamma$ in the neighbourhood of the point $z_0\in \bar{\C}$ where $G(z_0)=\infty$. Then by \eqref{invm} we can write
$$L(G(z))=\Gamma(z)^4M(\Gamma(z))\frac{\bar\Gamma_z(z)}{\Gamma_z(z)}.$$ We know that $\Gamma_z$ does not vanish (indeed, $\Gamma$ is the Gauss map of the image of $S$ by the isometry $\sigma\tau$, see Remark \ref{remg}), $M(0)=1/H$ and $\bar\Gamma_z$ has a finite limit when $z\to z_0$, so we can write
$$L(q)=\frac1{q^4}\varphi\left(\frac iq\right)$$
where $\varphi$ is bounded in a neighbourhood of $0$. Hence we can define $L(q)$ on the whole Riemann sphere $\bar\C$, and moreover the condition \eqref{condl} is satisfied. 

We want to prove next that $L(q)$ verifies \eqref{eqL}. For that, we divide \eqref{qzg} by $G_z$ and then differentiate with respect to $\bar{z}$. As $G$ verifies the Gauss map equation \eqref{eqg2}, $G_z\neq 0$ and $M$ satisfies \eqref{diffM}, we obtain the relation $$(L_q+2LA)G_{\bar z}+(L_{\bar q}+2LB+M\bar B)\bar G_{\bar z}=0$$ 
(where $L_q+2LA$ and $L_{\bar q}+2LB+M\bar B$ are evaluated at the point $G(z)$). So, using again \eqref{qzg} we have
$$(L_q+2LA)\bar L=(L_{\bar q}+2LB+M\bar B)\bar M$$
at every point $q=G(z)$. Since $G$ is a global diffeomorphism, this equation holds at every $q\in\bar\C$. 

Thus, we have proved the existence of a solution $L:\bar{\C}\flecha \C$ to \eqref{eqL} that satisfies the condition at infinity \eqref{condl}, so by Lemma \ref{lemq} this concludes the proof.
\end{proof}

Observe that by \eqref{qzg} it holds $$Q(G)=0.$$

We will now show that, in these conditions, this CMC $H$ sphere $S_H$ is unique (up to left translations) among immersed CMC $H$ spheres in $\sol$. For that, we shall use the following auxiliary results.

\begin{lem} \label{LleqM}
We have $$\left|\frac{L(q)}{M(q)}\right|<1$$ for all $q\in\bar\C$.
\end{lem}

\begin{proof}
When $q\in\C$, this is a consequence of \eqref{elege} and the fact that $G$ is an orientation-preserving diffeomorphism. When $q=\infty$, we use Remark \ref{remg}.
\end{proof}

\begin{lem} \label{localdiffeo}
Let $g:\Sigma\to\C$ be a nowhere antiholomorphic smooth map from a Riemann surface $\Sigma$ such that $Q(g)=0$. Then $g$ is a local diffeomorphism.
\end{lem}

\begin{proof}
Since $Q(g)=0$, we have $\bar g_z/g_z=-L(g)/M(g)$ and so $|g_z|>|\bar g_z|$ by Lemma \ref{LleqM}. 
\end{proof}

The following lemma states that solutions of the equation $Q(g)=0$ are locally unique up to a conformal change of parameter. 

\begin{lem}\label{unilem}
Let $g_1:\Sigma_1\to\C$ and $g_2:\Sigma_2\to \C$ be two nowhere antiholomorphic smooth maps from Riemann surfaces $\Sigma_1$, $\Sigma_2$. Assume that $Q(g_1)=0$ and $Q(g_2)=0$ (consequently, $g_1$ and $g_2$ are local diffeomorphisms by Lemma \ref{localdiffeo}). Assume that there exists an open set $U\subset\Sigma_2$ and a diffeomorphism $\varphi:U\to\varphi(U)\subset\Sigma_1$ such that $g_2=g_1\circ\varphi$ on $U$.  Then $\varphi$ is holomorphic.
\end{lem}


\begin{proof}
Since $g_{1z}$ and $g_{2z}$ do not vanish, the fact that $Q(g_1)=0$ and $Q(g_2)=0$ implies that
\begin{equation} \label{Qg1}
0=L(g_1)g_{1z}+M(g_1)\bar g_{1z},
\end{equation}
\begin{equation} \label{Qg2}
0=L(g_2)g_{2z}+M(g_2)\bar g_{2z}.
\end{equation}
We have
$$g_{2z}=(g_{1z}\circ\varphi)\varphi_z+(g_{1\bar z}\circ\varphi)\bar\varphi_z,$$
$$\bar g_{2z}=(\bar g_{1z}\circ\varphi)\varphi_z+(\bar g_{1\bar z}\circ\varphi)\bar\varphi_z.$$
Consequently, using \eqref{Qg2} and then \eqref{Qg1} at the point $\varphi(z)$, we obtain
$$(L(g_2)(g_{1\bar z}\circ\varphi)+M(g_2)(\bar g_{1\bar z}\circ\varphi))\bar\varphi_z=0.$$
Conjugating and applying again \eqref{Qg1} at the point $\varphi(z)$ gives
$$\frac{g_{1z}\circ\varphi}{M(g_2)}(|M(g_2)|^2-|L(g_2)|^2)\varphi_{\bar z}=0.$$
We have $g_{1z}\neq 0$ and $|M(g_2)|>|L(g_2)|$ by Lemma \ref{LleqM}. Hence we conclude that $\varphi_{\bar z}=0$.
\end{proof}

\begin{teo} \label{uniquenessdiffeo}
Let $H\neq 0$. Assume that there exists an immersed CMC $H$ sphere $S_H$ in $\sol$ whose Gauss map is a global diffeomorphism. Then, any other immersed CMC $H$ sphere $\Sigma$ in $\sol$ differs from $S_H$ at most by a left translation.
\end{teo}

\begin{proof}
We view $S_H$ as a conformal immersion from $\bar{\C}$ into $\sol$, whose Gauss map $G:\bar{\C}\flecha \bar{\C}$ is a global diffeomorphism, and $G_z\neq 0$ at every point. 

By Proposition \ref{existenceQ} and Theorem \ref{crt} we can conclude that if $\Sigma$ is another CMC $H$ sphere in $\sol$, and $g:\Sigma\equiv \bar{\C}\flecha \bar{\C}$ denotes its Gauss map, then the well-defined quadratic differential $Q(g)$ vanishes identically on $\Sigma$.

Nonetheless, we also have $Q(G)=0$. So, as $G$ is a global diffeomorphism, we have from Lemma \ref{unilem} (and real analyticity) that $g= G\circ \psi$ where $\psi:\bar{\C}\flecha \bar{\C}$ is a conformal automorphism of $\bar{\C}$. Thus, by conformally reparametrizing $\Sigma$ if necessary we have two conformal CMC $H$ immersions $X_1,X_2:\bar{\C}\flecha \sol$ with the same Gauss map $G$. From Theorem \ref{rep}, this implies that $X_1$ and $X_2$, i.e. $\Sigma$ and $S_H$, coincide up to a left translation, as desired.
\end{proof}

\begin{remark} \label{comment}
Our proof of Theorem \ref{uniquenessdiffeo} has two key ideas. One is to prove that,
associated to a sphere $S_H$ whose Gauss map is a global
diffeomorphism, we can construct a geometric differential $Q\,\rmd z^2$
for arbitrary surfaces in $\sol$ such that:
\begin{itemize}
\item[$(a)$] $Q\,\rmd z^2$ vanishes identically on $S_H$ (this is the meaning of \eqref{elege} and \eqref{qzg}), and
\item[$(b)$] $Q\,\rmd z^2$ vanishes identically \emph{only} on $S_H$ (this is Lemma \ref{unilem}).
\end{itemize}
This idea is very general, it does not use any differential equation, and could also
work in many other contexts. As a matter of fact, in this general
strategy the role of the Gauss map $g$ could be played by some other
geometric mapping into $\bar{\C}$ that determines the surface
uniquely.

The second key idea is to prove, using the Gauss map equation, that
this quadratic differential $Q\,\rmd z^2$ must actually vanish identically on any
sphere, as otherwise it would only have isolated zeros of negative
index, thus contradicting the Poincar\'e-Hopf theorem. This is also a
rather general strategy, that actually goes back to Hopf.

It is hence our impression that the underlying ideas in the proof of
Theorem \ref{uniquenessdiffeo} actually provide a new and flexible way of proving
Hopf-type theorems. 
\end{remark}

\begin{cor} \label{symmetries}
Let $S$ be a CMC sphere whose Gauss map is a diffeomorphism. Then there exists a point $p\in\sol$, which we will call the \emph{center} of $S$, such that $S$ is globally invariant by all isometries of $\sol$ fixing $p$. 

In particular, there exists two constants $d_1$ and $d_2$ such that $S$ is invariant by reflections with respect to the planes $x_1=d_1$ and $x_2=d_2$.
\end{cor}

\begin{proof}
Let $H$ be the mean curvature of $S$. By Theorem \ref{uniquenessdiffeo}, $S$ is the unique CMC $H$ sphere up to left translations. In particular, there exists a left translation
$$T:(x_1,x_2,x_3)\mapsto(e^{-c_3}(x_1+c_1),e^{c_3}(x_2+c_2),x_3+c_3)$$
such that $T\sigma(S)=S$, where $\sigma:(x_1,x_2,x_3)\mapsto(x_2,-x_1,-x_3)$. The isometry $T\sigma$ fixes the point $$p=\left(\frac{e^{-c_3}c_1+c_2}2,\frac{e^{c_3}c_2-c_1}2,\frac{c_3}2\right).$$

In the same way, there exists a left translation
$$T':(x_1,x_2,x_3)\mapsto(e^{-c'_3}(x_1+c'_1),e^{c'_3}(x_2+c'_2),x_3+c'_3)$$
such that $T'\tau(S)=S$, where $\tau:(x_1,x_2,x_3)\mapsto(-x_1,x_2,x_3)$. We have $T'\tau(x_1,x_2,x_3)=(e^{-c'_3}(-x_1+c'_1),e^{c'_3}(x_2+c'_2),x_3+c'_3)$, so since $S$ is compact we have $c'_3=0$ and $c'_2=0$. Now we have $(T\sigma T'\tau)^2(S)=S$, and
$$(T\sigma T'\tau)^2(x_1,x_2,x_3)=(x_1-c'_1+c_2+e^{-c_3}c_1,x_2+c_1+e^{c_3}(c_2-c'_1),x_3),$$ so again since $S$ is compact we get $c'_1=c_2+e^{-c_3}c_1$. Consequently, $T'\tau$ fixes $p$.

So $T\sigma$ and $T'\tau$ generate the isotropy group of $p$, which finishes the proof.
\end{proof}

\begin{remark}
Consider an equation of the form 
$$g_{z\bar z}=A(g)g_zg_{\bar z}+B(g)g_z\bar g_{\bar z},$$ where $A:\bar\C\to\C$ and $B:\bar\C\to\C$ are given functions satisfying a certain growth condition at $\infty$, so that the change of function $g\to i/g$ induces an equation of the same form. Then using the techniques above we can prove that it admits at most a unique (up to conformal change of parameter) solution $g:\bar\C\to\bar\C$ if the partial differential equations \eqref{diffM} and \eqref{eqL} admit solutions $L,M:\bar\C\to\C$ satisfying a certain growth condition at $\infty$ and such that $|L/M|<1$.
\end{remark}

\section{Existence and properties of CMC spheres} \label{existence}

The object of this section is to prove existence of index one CMC spheres and some of their properties. For this purpose we will use stability and nodal domain arguments. We will first recall some well known results of this context, and their application to CMC surfaces in $\sol$. A good reference is \cite{cheng} (see also \cite{fcinventiones,souam,rossman}).

\subsection{Preliminaries on the stability operator} \label{sec:prelstab}

Let $S$ be a compact CMC surface in $\sol$. Its \emph{stability operator} (or \emph{Jacobi operator}) is defined by
$$\cL=\Delta+||\cB||^2+\Ric(N)$$ where $\Delta$ is the Laplacian with respect to the induced metric, $\cB$ the second fundamental form of $S$, $N$ its unit normal vector field and $\Ric$ the Ricci curvature of $\sol$. The stability operator is also the linearized operator of the mean curvature functional.

The operator $-\cL$  admits a sequence
$$\lambda_1<\lambda_2\leqslant\lambda_3\leqslant\cdots\leqslant\lambda_k\leqslant\cdots$$
of eigenvalues such that $\lim_{k\to+\infty}\lambda_k=+\infty$. We will call $\lambda_k$ the $k$-th eigenvalue. Moreover the corresponding eigenspaces are orthogonal. The number of negative eigenvalues is called the \emph{index} of $S$ and will be denoted by $\ind(S)$. To the operator $-\cL$ is associated the quadratic form $\cQ$ defined by
\begin{equation} \label{quadratic}
\cQ(f,f)=-\int_Sf\cL f=\int_S||\grad f||^2-(||\cB||^2+\Ric(N))f^2.
\end{equation}
The quadratic form $\cQ$ acts naturally on the Sobolev space $\mathrm{H}^1(S)$.

A \emph{Jacobi function} is a function $f$ on $S$ such that $\cL f=0$. If $F$ is a Killing field in $\sol$, then $\langle N,F\rangle$ is a Jacobi function on $S$. Since $S$ is compact, it is not invariant by a one-parameter group of isometries; consequently the Jacobi fields $F_1$, $F_2$ and $F_3$ induce three linearly independent Jacobi functions. This implies that $0$ is an eigenvalue of $\cL$ with multiplicity at least $3$. Hence it cannot be the first eigenvalue, and so $\lambda_1<0$. Consequently, $\ind(S)\geqslant 1$, and $\ind(S)=1$ if and only if $\lambda_2=0$.

We say that $S$ is \emph{stable} (or \emph{weakly stable}) if
$$\cQ(f,f)=-\int_Sf\cL f\geqslant 0$$ for any smooth function $f$ on $S$ satisfying
$$\int_Sf=0.$$ If $S$ is stable, then $\ind(S)=1$ (indeed, if $\lambda_2<0$, then there exists a linear combination $f$ of the first two eigenfunctions such that $-\int_Sf\cL f<0$ and $\int_Sf=0$, which contradicts stability). Another well known property in this context is that solutions to the isoperimetric problem are stable compact embedded CMC surfaces (hence stable spheres in $\sol$).

We state the following well known theorems:
\begin{itemize}
\item Courant's nodal domain theorem: any eigenfunction of $\lambda_2$ has at most $2$ nodal domains (Proposition 1.1 in \cite{cheng}),
\item if $S$ is a sphere, then the dimension of the eigenspace of $\lambda_2$ is at most $3$ (Theorem 3.4 in \cite{cheng}).
\end{itemize}
The results in \cite{cheng} are stated for the Laplacian but can be extended for an operator of the form $\Delta+V$ where $V$ is a function (see for instance \cite{rossman} for a proof of Courant's nodal domain theorem; then the other results are deduced using topological arguments and the local behaviour of solutions of $\Delta f+Vf=0$). Also the convention for the numbering of the eigenfunctions in \cite{cheng} differs from ours: what is called the $k$-th eigenvalue in \cite{cheng} is what we call the $(k+1)$-th eigenvalue.

In particular, if $S$ is an index one compact CMC surface in $\sol$, then any Jacobi function on $S$ admits at most two nodal domains (see also Proposition 1 in \cite{roswillmore}).

\subsection{The Gauss map of index one CMC spheres}

\begin{pro} \label{diffeo}
Let $S$ be an index one CMC sphere. Then its Gauss map $g:S\to\bar\C$ is an orientation-preserving diffeomorphism. 
\end{pro}

\begin{proof}
Since $g$ is a map from a sphere to a sphere, it suffices to prove that its Jacobian $J(g)=|g_z|^2-|g_{\bar z}|^2$ is positive everywhere.

Let $F=a_1F_1+a_2F_2+a_3F_3$ be a non-zero Killing field. Then
 \begin{equation}\label{Jacin}
 f=\langle F,N\rangle=(a_1-a_3x_1)e^{x_3}N_1+(a_2+a_3x_2)e^{-x_3}N_2+a_3N_3
  \end{equation}
is a Jacobi function. Since $\ind(S)=1$, the function $f$ has at most two nodal domains, so since $S$ is a sphere, $f$ cannot have a singular point on the nodal curve (see also Theorem 3.2 in \cite{cheng}). Consequently, if $f(p)=0$ for some $p$, then $f_z(p)\neq 0$.

We have 
\begin{eqnarray*}
f_z & = & (a_1-a_3x_1)e^{x_3}N_{1z}+(a_2+a_3x_2)e^{-x_3}N_{2z}+a_3N_{3z} \\
& & +(a_1-a_3x_1)e^{x_3}x_{3z}N_1-a_3e^{x_3}x_{1z}N_1 \\
& & -(a_2+a_3x_2)e^{-x_3}x_{3z}N_2+a_3e^{-x_3}x_{2z}N_2.
\end{eqnarray*}
Setting 
 \begin{equation}\label{Jacin2}
 b=(a_1-a_3x_1)e^{x_3}+i(a_2+a_3x_2)e^{-x_3}, \end{equation}
we get
\begin{equation}\label{Jacin3}
\def\arraystretch{2}\begin{array}{lll}
f_z & = & \displaystyle \re b\frac{\partial}{\partial z}\frac{g+\bar g}{1+|g|^2}
+\im b\frac{\partial}{\partial z}\frac{i(\bar g-g)}{1+|g|^2}
+a_3\frac{\partial}{\partial z}\frac{1-|g|^2}{1+|g|^2} \\
& &  \displaystyle+\frac{bg+\bar b\bar g}{1+|g|^2}A_3-a_3A_1N_1+a_3A_2N_2 \\
& = & \displaystyle \frac{\bar bg_z+b\bar g_z}{1+|g|^2}-\frac{\bar bg+b\bar g}{(1+|g|^2)^2}(\bar gg_z+g\bar g_z)
-\frac{2a_3}{(1+|g|^2)^2}(\bar gg_z+g\bar g_z) \\
& & \displaystyle +\frac{bg+\bar b\bar g}{1+|g|^2}\frac{\eta}2+a_3\frac{\eta}{2\bar g}\frac{g-\bar g^3}{1+|g|^2}.
\end{array}
\end{equation}

%
%

Let $z_0\in S$. Assume first that $|g(z_0)|\neq 1$. We can also assume that $g(z_0)\neq\infty$, since $g(z_0)=\infty$ can be dealt with in the same way using the isometry $\sigma\tau$ (see Remark \ref{rem1}). Then $f(z_0)=0$ if and only if
$$a_3=-\frac{\bar bg+b\bar g}{1-|g|^2}$$ at $z_0$. If this condition is satisfied, then, at the point $z_0$, using \eqref{formulaeta} we get
\begin{eqnarray*}
(1+|g|^2)^2f_z & = & (1+|g|^2)(\bar bg_z+b\bar g_z)-(\bar bg+b\bar g)(\bar gg_z+g\bar g_z)
+2\frac{\bar bg+b\bar g}{1-|g|^2}(\bar gg_z+g\bar g_z) \\
& & +(1+|g|^2)(bg+\bar b\bar g)\frac{2\bar gg_z}{R(g)}-\frac{\bar bg+b\bar g}{1-|g|^2}(1+|g|^2)\frac{2g_z}{R(g)}(g-\bar g^3) \\
& = & (1+|g|^2)(\bar bg_z+b\bar g_z)+\frac{1+|g|^2}{1-|g|^2}(\bar bg+b\bar g)(\bar gg_z+g\bar g_z) \\
& & +\frac{2g_z}{R(g)}\frac{1+|g|^2}{1-|g|^2}(\bar b+b\bar g^2)(\bar g^2-g^2),
\end{eqnarray*}
and so
$$(1-|g|^4)f_z=C_1b+C_2\bar b$$ with
$$C_1:=\left(1+\frac{2(\bar g^2-g^2)}{R(g)}\right)\bar g^2 g_z+\bar g_z
=\frac{\overline{R(g)}}{R(g)}\bar g^2 g_z+\bar g_z,$$
$$C_2:=\left(1+\frac{2(\bar g^2-g^2)}{R(g)}\right)g_z+g^2\bar g_z
=\frac{\overline{R(g)}}{R(g)}g_z+g^2\bar g_z.$$
From this we get $C_1b+C_2\bar b\neq0$. This holds for all non-zero Killing fields $F$ such that $\langle F,N\rangle$ vanishes at $z_0$, i.e., for all $b\in\C\setminus\{0\}$. This means that the real-linear equation $C_1b+C_2\bar b=0$ in $(b,\bar b)$ has $b=0$ as unique solution, so we get
$|C_1|^2\neq|C_2|^2$, which is equivalent to $$|g_z|^2\neq|\bar g_z|^2.$$

Assume now that $|g(z_0)|=1$ at some $z_0\in S$ where $f(z_0)=0$. Then $N_3(z_0)=0$ and from \eqref{Jacin}, \eqref{Jacin2} we have (at $z_0$) $\re(b\bar g)=0$, that is, $$b=i\mu g$$ for some $\mu \in \R$. Hence, using \eqref{formulaeta} and the fact that $|g(z_0)|=1$, equation \eqref{Jacin3} simplifies (at $z_0$) to
$$f_z = -\frac{(a_3 +i\mu) \bar{g}}{2} g_z - \frac{(a_3-i\mu)g}{2} \bar{g}_z + \frac{(a_3+ i \mu) (g-\bar{g}^3)}{R(g)} g_z,$$ 
and after some computations to
$$f_z=-g_z(a_3+i\mu)\frac{\overline{R(g)}}{R(g)}\frac{\bar g}2-\bar g_z(a_3-i\mu)\frac g2.$$
So, as $f_z (z_0)\neq 0$, this quantity must be non-zero for all non-zero Killing fields $F$ such that $\langle F,N\rangle$ vanishes at $z_0$, i.e., for all $a_3 + i\mu \in \C\setminus \{0\}.$ Again, this means that $|g_z| (z_0) \neq |\bar{g}_z| (z_0)$. 


We have proved that $|g_z|^2\neq|\bar g_z|^2$ everywhere, i.e., that $g$ is a diffeomorphism. To prove that $g$ preserves orientation, it suffices to prove that $|g_z|^2>|\bar g_z|^2$ at some point. At a point of $S$ where $x_3$ has an extremum, we have $g=0$ or $g=\infty$, so since $g$ is a diffeomorphism, $x_3$ admits exactly one extremum where $g=0$, at a point that we will denote $z_0$.

We have $(x_3)_z=\eta/2=2\bar gg_z/R(g)$. Hence, at the point $z_0$, using that $g(z_0)=0$ and $R(0)=H$, we get
$$(x_3)_{zz}=\frac{2g_z\bar g_z}H,\quad
(x_3)_{z\bar z}=\frac{2|g_z|^2}H,\quad
(x_3)_{\bar z\bar z}=\frac{2g_{\bar z}\bar g_{\bar z}}H.$$
Since $x_3$ has an extremum at $z_0$, the Hessian of $x_3$ at $z_0$ has a non-negative determinant, i.e., 
$4(x_3)_{z\bar z}^2-4(x_3)_{zz}(x_3)_{\bar z\bar z}\geqslant 0$. Since since $g_z\neq 0$, we obtain that $|g_z|^2\geqslant|\bar g_z|^2$ at $z_0$, which concludes the proof.
\end{proof}

\subsection{Bounds on the second fundamental form and the diameter}

\begin{pro}
Let $H\neq 0$ and let $S$ be a CMC $H$ sphere whose Gauss map is an orientation-preserving diffeomorphism. Then its second fundamental form $\cB$ satisfies
\begin{equation} \label{normB}
||\cB||^2<4H^2+4|H|+2.
\end{equation}
\end{pro}

\begin{proof}
By a classical computation we have
$$\det\cB=H^2-\frac{4|P|^2}{\lambda^2}$$ where $P\rmd z^2$ is the Hopf differential, and so
$$||\cB||^2=4H^2-2\det\cB=2H^2+\frac{8|P|^2}{\lambda^2}.$$ 
By \eqref{landa}, \eqref{formulaeta} and Remark \ref{formulaP} we have
$$\frac{|P|}{\lambda}=\frac1{2(1+|g|^2)^2}\left|R(g)\frac{\bar g_z}{g_z}-1+\bar g^4\right|.$$
Since $g$ is an orientation-preserving diffeomorphism we have $|\bar g_z|<|g_z|$ and so
$$\frac{|P|}{\lambda}<\frac{|R(g)|+1+|g|^4}{2(1+|g|^2)^2}
\leqslant\frac{|H|+1}2,$$  which gives the result (notice that the first inequality is strict since $R(g)$ does not vanish when $H\neq0$).
\end{proof}

The following diameter estimate when $H>1/\sqrt3$ is a consequence of a theorem of Rosenberg \cite{rosenbergaustralian}. We also refer to \cite{fcinventiones} for some arguments used in the proof.

\begin{lem} \label{diameter}
Let $H>1/\sqrt3$ and $S$ be an index one CMC $H$ sphere. Then its intrinsic diameter is less than or equal to $8\pi/\sqrt{3(3H^2-1)}$.
\end{lem}

\begin{proof}
We recall that a domain $U\subset S$ is said to be \emph{strongly stable} if $-\int_Uf\cL f\geqslant 0$ for all smooth functions $f$ on $U$ that vanish on $\partial U$.

Fix $p_0\in S$ and for $r>0$ let $B_r$ denote the geodesic ball of radius $r$ centered at $p_0$. Let
$$\rho:=\sup\{r>0;B_r\text{ is strongly stable}\}.$$ We have $\rho<+\infty$ since $S$ is not strongly stable.

Let $\varepsilon>0$ and $U=S\setminus B_{\rho+\varepsilon}$. We claim that $U$ is strongly stable (when $U\neq\emptyset$).

Indeed, assume that $U$ is not strongly stable. Then the first eigenvalue $\lambda_1(U)$ of $-\cL$ (for the Dirichlet problem with zero boundary condition) on $U$ is negative; let $f_1$ be an eigenfunction on $U$ for $\lambda_1(U)$, extended by $0$ on $B_{\rho+\varepsilon}$. In the same way, $B_{\rho+\varepsilon}$ is not strongly stable and so $\lambda_1(B_{\rho+\varepsilon})<0$; let $f_2$ be an eigenfunction on $B_{\rho+\varepsilon}$ for $\lambda_1(B_{\rho+\varepsilon})$, extended by $0$ on $U$. Then $f_1$ and $f_2$ have supports that overlap on a set of measure $0$, so they are orthogonal for $\cQ$; then $\cQ$ is negative definite on the space spanned by $f_1$ and $f_2$, which has dimension $2$. This contradicts the fact that $\ind(S)=1$. This proves the claim.

Since $\sol$ has scalar curvature $-2$ and since $H>1/\sqrt3$, Theorem 1 in \cite{rosenbergaustralian} states that if $\Omega$ is a strongly stable domain in $S$, then for all $p\in\Omega$ we have
\begin{equation} \label{distineq}
\dist_S(p,\partial\Omega)\leqslant\frac{2\pi}{\sqrt{3(3H^2-1)}}
\end{equation}
where $\dist_S$ denotes the intrinsic distance on $S$. (Observe that in \cite{rosenbergaustralian} the scalar curvature is defined as one half of the trace of the Ricci tensor, whereas in our convention it is defined as the trace of the Ricci tensor.) Applying \eqref{distineq} for $p_0$ in $B_\rho$ we get 
$$\rho\leqslant\frac{2\pi}{\sqrt{3(3H^2-1)}}.$$

Let $p\in S$. We claim that
$$\dist_S(p,p_0)\leqslant\frac{4\pi}{\sqrt{3(3H^2-1)}}+\varepsilon.$$ If $p\in B_{\rho+\varepsilon}$, this is clear. So assume that $p\in U$ (in the case where $U\neq\emptyset$). Then applying \eqref{distineq} for the strongly stable domain $U$ we get
$$\dist_S(p,\partial U)\leqslant\frac{2\pi}{\sqrt{3(3H^2-1)}}.$$ But $\partial U=\partial B_{\rho+\varepsilon}$, so
$$\dist_S(p,p_0)\leqslant\frac{2\pi}{\sqrt{3(3H^2-1)}}+\rho+\varepsilon.$$
This finishes proving the claim.

Since this holds for all $\varepsilon>0$, the lemma is proved.
\end{proof}

\subsection{Embeddedness} \label{sec:embeddedness}

\begin{pro} \label{embeddedness}
Let $S$ be an immersed CMC sphere whose Gauss map is a diffeomorphism. Then $S$ is embedded, and consequently it is a symmetric bigraph in the $x_1$ and $x_2$ directions.
\end{pro}

\begin{proof}
We view $S$ as a conformal immersion $X:S\equiv\bar\C\to\sol$. By Corollary \ref{symmetries}, up to a left translation, $X(S)$ is symmetric with respect to the planes $\Pi_1=\{x_1=0\}$ and $\Pi_2=\{x_2=0\}$. It will be important to keep in mind that the frame $(\frac{\partial}{\partial x_1},\frac{\partial}{\partial x_2},\frac{\partial}{\partial x_3})$ is orthogonal.

Let $N=[N_1,N_2,N_3]:S\to\rmT\sol$ be the unit normal vector field of $X$ and $g:S\to\bar\C$ the Gauss map of $S$. We first notice that at a critical point of $x_3$ we have $N=\pm E_3$, i.e., $g=0$ or $g=\infty$. Hence, since $g$ is a diffeomorphism, $x_3$ admits exactly one minimum at a point $p\in S$, one maximum at a point $q\in S$ and no other critical point. Since $X(S)$ is symmetric with respect to the planes $\Pi_1$ and $\Pi_2$, we necessarily have $X(p)\in\Pi_1$, $X(q)\in\Pi_1$, $X(p)\in\Pi_2$, $X(q)\in\Pi_2$. 

Let $\gamma$ be the set of the points of $S$ where $N$ is orthogonal to $E_1$. Then $\gamma=\{z\in S;g(z)\in i\R\cup\{\infty\}\}$. Since $g$ is a diffeomorphism, $\gamma$ is a closed embedded curve in $S$.

Let $z_0\in\gamma$. Since $X(S)$ is invariant by the symmetry $\tau$ with respect to $\Pi_1$, there exists a point $z_1\in S$ such that $X(z_1)=\tau(X(z_0))$ and $N(z_1)=[-N_1(z_0),N_2(z_0),N_3(z_0)]$. But $z_0\in\gamma$ so $N_1(z_0)=0$ and so $g(z_1)=g(z_0)$. Since $g$ is a diffeomorphism we have $z_1=z_0$ and so $X(z_0)\in\Pi_1$. This proves that $X(\gamma)\subset\Pi_1$.

We claim that $X(\gamma)$ is embedded and is a bigraph in the $x_2$ direction.

We have $p,q\in\gamma$ and, as explained before, $X(p),X(q)\in\Pi_2$. The set $\gamma\setminus\{p,q\}$ consists of the following two disjoint curves:
$$\gamma^+:=\{z\in S;g(z)=ir,r\in(0,+\infty)\},\quad\gamma^-:=\{z\in S;g(z)=ir,r\in(-\infty,0)\}.$$
Then the curve $X(\gamma^+)$ is transverse to $\frac\partial{\partial x_2}$ since $N$ is never parallel to $E_3$ on $\gamma^+$. Consequently, $X(\gamma^+)$ can be written as a graph $\{(x_2,x_3);x_3\in I,x_2=f(x_3)\}$ where $I=(x_3(p),x_3(q))$ and $f:I\to\R$ is a smooth function such that $\lim_{t\to x_3(p)}f(t)=\lim_{t\to x_3(q)}f(t)=0$. By symmetry, $X(\gamma^-)$ is the graph $\{(x_2,x_3);x_3\in I,x_2=-f(x_3)\}$.

We now prove that $f>0$ or $f<0$. Without loss of generality, we may assume that a maximum of $x_2$ on $\gamma$ is attained at a point of $\gamma^+$, and consequently a minimum of $x_2$ on $\gamma$ is attained at a point of $\gamma^-$. But at a critical point of $x_2$ on $\gamma$ we have $N=\pm E_2$, i.e., $g=\pm i$. Since $g$ is a diffeomorphism, $x_2$ admits exactly two critical point on $\gamma$, so exactly one on $\gamma^+$, which is the abovementioned maximum. This implies that $f>0$.

This finishes proving the claim. 

Consequently, $X(\gamma)$ is embedded and bounds a domain $U$ in $\Pi_1$; this domain $U$ is topologically a disk.
Let $$S^+:=\{z\in S;g(z)\in\C,\re(g(z))>0\},\quad S^-:=\{z\in S;g(z)\in\C,\re(g(z))<0\}.$$ Then $S$ is the disjoint union of $S^+$, $S^-$ and $\gamma$. Also, $X(S^+)$ is transverse to $\frac\partial{\partial x_1}$ since $N$ is never orthogonal to $E_1$ on $S^+$, and is bounded by $X(\gamma)=\partial U$. Consequently, $X(S^+)$ can be written as a graph $\{(x_1,x_2,x_3);(x_2,x_3)\in U,x_1=h(x_2,x_3)\}$ where $h:U\to\R$ is a smooth function. By symmetry, $X(\gamma^-)$ is the graph $\{(x_1,x_2,x_3);(x_2,x_3)\in U,x_1=-h(x_2,x_3)\}$. By the same argument as above, we prove that $h>0$ or $h<0$ (otherwise $x_1$ would admit at least $3$ critical points). So $X(S)$ is embedded.
\end{proof}

\subsection{Deformations of CMC spheres}

\begin{lem}
Let $\Sigma$ be an oriented embedded compact surface (not necessarily CMC) in $\sol$. Let $N$ be its unit normal vector and ${\mathcal H}$ its mean curvature function. Then, for all Killing field $F$ of $\sol$, it holds
\begin{equation} \label{stokes}
\int_\Sigma\langle F,N\rangle=0,\quad \int_\Sigma{\mathcal H}\langle F,N\rangle=0.
\end{equation}
\end{lem}

\begin{proof}
These formulas are well-known. We include a proof for the convenience of the reader.

Since Killing fields have divergence zero, the first formula follows from Stokes' formula applied on the compact region bounded by $\Sigma$.

Let $J:\rmT\Sigma\to\rmT\Sigma$ be the rotation of angle $\pi/2$. Denote by $\hat\nabla$, $\nabla$ and $\cB$ the Riemannian connection of $\sol$, the Riemannian connection of $\Sigma$ and the second fundamental form of $\Sigma$ respectively. We define on $\Sigma$ the $1$-form $\omega$ by
$$\forall v\in\rmT\Sigma,\quad\omega(v)=\langle F,Jv\rangle.$$

Let $(e_1,e_2)$ be a local direct orthonormal frame on $\Sigma$ (hence $Je_1=e_2$ and $Je_2=-e_1$). Then we have
\begin{eqnarray*}
\rmd\omega(e_1,e_2) & = & e_1\cdot\omega(e_2)-e_2\cdot\omega(e_1)-\omega([e_1,e_2]) \\
& = & -e_1\cdot\langle F,e_1\rangle-e_2\cdot\langle F,e_2\rangle
-\langle F,J(\nabla_{e_1}e_2-\nabla_{e_2}e_1)\rangle \\
& = & -\langle F,\hat\nabla_{e_1}e_1\rangle-\langle F,\hat\nabla_{e_2}e_2\rangle
+\langle F,\nabla_{e_1}e_1\rangle+\langle F,\nabla_{e_1}e_1\rangle \\
& = & -\langle F,\cB(e_1,e_1)N+\cB(e_2,e_2)N\rangle \\
& = & -2{\mathcal H}\langle F,N\rangle.
\end{eqnarray*}
Here we used the fact that $J$ commutes with $\nabla$ and the fact that $\langle\hat\nabla_vF,v\rangle=0$ for all $v$ since $F$ is a Killing field. Then Stokes' formula yields
$$0=\int_\Sigma\rmd\omega=-2\int_\Sigma{\mathcal H}\langle F,N\rangle.$$
\end{proof}

In the next propostion we will deform index one CMC spheres using the implicit function theorem and the fact that all the Jacobi functions on such a sphere come from ambient Killing fields. 

\begin{pro} \label{localdeformation}
Let $S$ be an index one CMC $H$ sphere. Then there exist $\varepsilon>0$ and a real analytic family $(S_H)_{H\in(H_0-\varepsilon,H_0+\varepsilon)}$ of spheres such that $S_{H_0}=S$ and $S_H$ has CMC $H$ for all $H\in(H_0-\varepsilon,H_0+\varepsilon)$. Also, if $H\in(H_0-\varepsilon,H_0+\varepsilon)$ and $\tilde S$ is a CMC $H$ sphere close enough to $S$, then $\tilde S=S_H$ (up to a left translation).

Moreover, the spheres $S_H$, $H\in(H_0-\varepsilon,H_0+\varepsilon)$, have index one.
\end{pro}

\begin{proof}
By Propositions \ref{diffeo} and \ref{embeddedness}, the Gauss map of $S$ is a diffeomorphism and $S$ is embedded. We view $S$ as an embedding $X:S\to\sol$ and we will identify $S$ and $X(S)$. 

Let us start by proving the existence of the family $(S_H)$ by means of a classical deformation technique (see for instance \cite{prindiana,koiso} for details).

Let $\alpha>0$. For a function $\varphi\in\rmC^{2,\alpha}(S)$, we define the normal variation $X_{\varphi}:S\to\sol$ by $$X_{\varphi}(s)=\exp_{X(s)}(\varphi(s)N(s)),$$
where $\exp$ denotes the exponential map for the Riemannian metric on $\sol$ and $N$ the unit normal vector field of $S$. There exists some neighbourhood $\Omega$ of $0$ in $\rmC^{2,\alpha}(S)$ such that $X_{\varphi}$ is an embedding for all $\varphi\in\Omega$; then let $\cH(\varphi)$ be the mean curvature of $X_{\varphi}$. Hence we have defined a map
$$\cH:\Omega\to\rmC^{0,\alpha}(S),$$ and the differential of $\cH$ at $0\in\rmC^{2,\alpha}(S)$ is one half the stability operator of $S$:
$$\rmd_0\cH=\frac12\cL:\rmC^{2,\alpha}(S)\to\rmC^{0,\alpha}(S).$$ 

If $E\subset\rmC^{0,\alpha}(S)$, we will let $E^\perp$ denote the orthogonal of $E$ in $\rmC^{0,\alpha}(S)$ for the $\mathrm{L}^2(S)$ scalar product. As explained in Section \ref{sec:prelstab}, all Jacobi functions on $S$ come from ambient Killing fields, i.e., $\ker\cL$ has dimension $3$ and is generated by the functions $f_k:=\langle F_k,N\rangle$, $k=1,2,3$, where $(F_1,F_2,F_3)$ is the basis of Killing fields defined in Section \ref{sol}. 
We also have $\im\cL=(\ker\cL)^\perp$.

We now consider the map
$$\begin{array}{ll}
\cK: & \Omega\times\R^3 \to\rmC^{0,\alpha}(S) \\
& (\varphi,a_1,a_2,a_3)\mapsto\cH(\varphi)-a_1\langle F_1,N_\varphi\rangle
-a_2\langle F_2,N_\varphi\rangle-a_3\langle F_3,N_\varphi\rangle\end{array}$$
where $N_\varphi$ denotes the unit normal vector field of $X_\varphi$. Then we have
$$\begin{array}{ll}
\rmd_{(0,0,0,0)}\cK=\Phi: & \rmC^{2,\alpha}(S)\times\R^3 \to\rmC^{0,\alpha}(S) \\
& (u,b_1,b_2,b_3)\mapsto\frac12\cL u-b_1f_1-b_2f_2-b_3f_3.\end{array}$$

We first claim that $\Phi$ is surjective. Let $v\in\rmC^{0,\alpha}(S)$. If $v\in\ker\cL$, then there exists $(b_1,b_2,b_3)$ such that $v=\Phi(0,b_1,b_2,b_3)$. If $v\in(\ker\cL)^\perp$, then there exists $u\in\rmC^{2,\alpha}(S)$ such that $\cL u=2v$, and so $v=\Phi(u,0,0,0)$. This proves the claim.

Also, since $\im\cL=(\ker\cL)^\perp$, we have $\ker\Phi=\ker\cL\times\{(0,0,0)\}$, so $\ker\Phi$ has dimension $3$.

We can now apply the implicit function theorem (see for instance \cite{lang}): there exist $\varepsilon>0$, a real analytic family $(\varphi_H)_{H\in(H_0-\varepsilon,H_0+\varepsilon)}$ of functions and real analytic functions $a_1,a_2,a_3:(H_0-\varepsilon,H_0+\varepsilon)\to\R$ such that, for all $H\in(H_0-\varepsilon,H_0+\varepsilon)$,
$$\cK(\varphi_H,a_1(H),a_2(H),a_3(H))=H$$
and $\cK^{-1}(H)$ is locally a manifold of dimension $3$.

Now, applying \eqref{stokes} to $S_H:=X_{\varphi_H}(S)$, we get $a_1(H)=a_2(H)=a_3(H)=0$, and so $\cH(\varphi_H)=H$.
This gives the existence of the family $(S_H)$ as announced. Also, in a neighbourhood of $S_H$ the set of CMC $H$ spheres is a $3$-dimensional manifold, so it is composed uniquely of images of $S_H$ by left translations. This gives the announced uniqueness result.


Let us show next that all spheres in this deformation have index one. Assume that there is some $H_1\in(H_0-\varepsilon,H_0+\varepsilon)$ such that $\ind(S_H)\neq 1$. Without loss of generality, we can assume that $H_1<H_0$. Let 
$$\hat H=\inf\{H>H_1;\ind(S_H)=1\}.$$ 
The functions $\lambda_k(S_H)$ and their eigenspaces are continuous with respect to $H$ \cite{kato,kodaira}. Consequently, $\lambda_2(S_{\hat H})=0$, $H_1<\hat H$, $\lambda_2(S_H)<0$ for all $H\in[H_1,\hat H)$ and there exists a continuous family of functions $$(f_H)_{H\in[H_1,\hat H]}$$ such that $f_H$ is an eigenfunction of $\lambda_2(S_H)$, $\int_{S_H}f_H^2=1$ and $f_H$ is orthogonal to the Jacobi functions on $S_H$ coming from ambient Killing fields. This implies in particular that $\lambda_2(S_{\hat H})$ has multiplicity at least $4$, which contradicts Theorem 3.4 in \cite{cheng}. 
\end{proof}

\begin{pro} \label{family}
There exists a real analytic family
$$(S_H)_{H>1/\sqrt3}$$ of spheres such that, for all $H>1/\sqrt3$, $S_H$ has CMC $H$ and $\ind(S_H)=1$.
\end{pro}

\begin{proof}
Let $S$ be a solution of the isoperimetric problem that has CMC $H_0$ with $H_0>1/\sqrt3$ (such a solution exists since there exist solutions with arbitrary large mean curvature). Then $S$ is a compact embedded CMC surface, hence it is a sphere by the Alexandrov reflection argument. Moreover, $S$ is stable and so has index one. Then by Proposition \ref{localdeformation} there exists a local deformation $(S_H)_{H\in(H_0-\varepsilon,H_0+\varepsilon)}$ for some $\varepsilon>0$.

Let $(H^-,H^+)$ be the largest interval to which we can extend the deformation $(S_H)$. By the same argument as that of the proof of Proposition \ref{localdeformation}, all the spheres $S_H$ have index one. Consequently, by Propositions \ref{diffeo} and \ref{embeddedness}, they are embedded and symmetric bigraphs in the $x_1$ and $x_2$ directions.

We first notice that $H^-\geqslant0$ since there is no compact minimal surface in $\sol$.
We will prove that $H^-\leqslant1/\sqrt3$ and $H^+=+\infty$. 

Assume that $H^->1/\sqrt3$. Let $(H_n)$ be a sequence converging to $H^-$ and such that $H_n>H^-$ for all $n$. Then by \eqref{normB} the second fundamental forms of the spheres $S_{H_n}$ are uniformly bounded. Also, by Lemma \ref{distineq} the diameters of the spheres $S_{H_n}$ are bounded; thus the areas of the spheres $S_{H_n}$ are bounded (by the Rauch comparison theorem), so in particular the spheres $S_{H_n}$ satisfy local uniform area bounds (see for instance \cite{ritoreros}). Without loss of generality we will assume that all spheres $S_{H_n}$ contain a certain fixed point. Consequently, by standard arguments \cite{prmartina,mt} there exists a subsequence, which will also be denoted $(H_n)$, such that the spheres $S_{H_n}$ converge to a properly weakly embedded CMC $H^-$ surface $S$, and $S$ must be compact since the diameters of the spheres $S_{H_n}$ are bounded. (The convergence is the convergence in the $\mathrm{C}^k$ topology, for every $k\in\N$, as local graphs over disks in the tangent planes with radius independent of the point.) Since $S$ is a limit of bigraphs in the $x_1$ and $x_2$ directions, it is a sphere (and the convergence is with multiplicity one). Finally, $S$ has index one since the spheres $S_{H_n}$ have index one (having index one is equivalent to $\lambda_2=0$).

Thus we can apply Proposition \ref{localdeformation} to $S$: we obtain the existence of a unique deformation $(\tilde S_H)_{H\in(H^--\eta,H^-+\eta)}$ for some $\eta>0$. Since the spheres $S_{H_n}$ converge to $S$, there exists some integer $m$ such that $S_{H_m}=\tilde S_{H_m}$, and so we have $S_H=\tilde S_H$ for $H$ in some neighbourhood of $H_m$ (here the equalities between spheres are always up to a left translation). This means that we can analytically extend the family $(S_H)$ for $H\leqslant H^-$ by setting $S_H:=\tilde S_H$ for all $H\in(H^--\eta,H^-]$. This contradicts the fact that $(H^-,H^+)$ is the largest interval for which the $S_H$ are defined.

Hence we have proved that $H^-\leqslant1/\sqrt3$. The proof that $H^+=+\infty$ is similar (indeed, if $H^+$ is finite, then we have the bound on the second fundamental form).
\end{proof}

\begin{remark}
Even if $S_{H_0}$ is stable, it is not clear whether the $S_H$ are stable or not. For instance, in $\s^2\times\R$, Souam \cite{souam} proved that among the one parameter family of rotational CMC spheres, the ones with mean curvature less than some constant are unstable.
\end{remark}

\section{Conclusion and final remarks} \label{sec:conclusion}

We can now summarize and conclude the proofs of the main theorems.

\begin{proof1}
In Theorem \ref{hopfmain}, the implications $(a)\Rightarrow(b)\Rightarrow(c)$ are explained in Section \ref{sec:prelstab} and the implication $(c)\Rightarrow(d)$ is Proposition \ref{diffeo}. The embeddedness of $\Sigma_H$ is Proposition \ref{embeddedness}. The uniqueness among immersed spheres is Theorem \ref{uniquenessdiffeo}, and the uniqueness among compact embedded surfaces is then obtained thanks to the Alexandrov reflection technique. Finally, Theorem \ref{main} is a corollary of Theorem \ref{hopfmain} and Proposition \ref{family}.
\end{proof1}

Let us now focus on some consequences for the isoperimetric problem in $\sol$. First, index one spheres constitute a finite or countable number of real analytic families $(S_H)_{H\in(H^-,H^+)}$ parametrized by mean curvature belonging to pairwise disjoint intervals. Hence the solutions to the isoperimetric problem belong to these families. Moreover, in such a family, if $A(H)$ and $V(H)$ denote respectively the area of $S_H$ and the volume enclosed by $S_H$, then $A'(H)=2HV'(H)$ and the sphere $S_{H_0}$ is stable if and only if $V'(H_0)\leqslant 0$ (this was proved in \cite{souam} for $\S^2\times\R$ but the proof readily extends to $\sol$).

\begin{pro} 
The following statements hold in $\sol$.
\begin{itemize}
\item[$(a)$] Every solution to the isoperimetric problem is globally invariant by \emph{all} isometries fixing some point.
\item[$(b)$] Two different solutions (up to left translations) to the isoperimetric problem cannot have the same mean curvature.
\end{itemize}
\end{pro}

\begin{proof}
Item $(a)$ is an immediate consequence of Corollary \ref{symmetries} and Proposition \ref{diffeo}. Item $(b)$ is a particular case of Theorem \ref{hopfmain}.
\end{proof}

However, we do not know if, for a given volume, there can exist several solutions (hence with different mean curvatures). Item $(a)$ states that the solutions to the isoperimetric problem are as symmetric as possible; in \cite{nardulli} this was proved for small volumes in a general compact manifold.

\begin{remark}
One can prove using Propostion \ref{localdeformation} that for a given volume there exists a finite number of solutions to the isoperimetric problem. Using moreover the above characterization of stable spheres, one can prove that the isoperimetric profile is concave.
\end{remark}

The value $1/\sqrt3$ in Theorem \ref{main} does not seem optimal and does not seem to have a geometric signification, as we explain next. 

In $\H^3$, the value $1$ is a special value for mean curvature, in the sense that CMC $H$ surfaces in $\H^3$ have very different behaviours if $|H|>1$, $|H|=1$ or $|H|<1$ (for instance, compact CMC $H$ surfaces in $\H^3$ exist if and only if $|H|>1$). In the same way, $1/2$ is a special value for mean curvature in $\H^2\times\R$.

On the contrary, in $\sol$ there exist CMC $H$ spheres with $H\neq0$ and $|H|$ arbitrarily small. Also, in no equation regarding conformal immersions is the fact that $H>1/\sqrt3$ important. This number $1/\sqrt3$ only appeared in the Diameter Lemma \ref{diameter}, which relies on a stability argument: the fact that $H>1/\sqrt3$ implies that some quantity in the Jacobi operator is necessarily positive. However it is only a \emph{sufficient} condition.
The same kind of problems have appeared previously in related questions: for instance, it is proved in \cite{souam} that a stable compact CMC $H$ surface is a sphere if $H>1/\sqrt2$, but this value of $H$ is conjectured not to be optimal.

This is why it is natural to propose the following conjecture.

\begin{conj}
For every $H>0$ there exists an embedded CMC $H$ sphere $S_H$, which is the unique (up to left translations) immersed CMC $H$ sphere and the unique (up to left translations) embedded compact CMC $H$ surface. Moreover, $S_H$ is a solution to the isoperimetric problem, and the family $(S_H)_{H>0}$ is real analytic.
\end{conj}

\bibliographystyle{plain}
\bibliography{spheres}

\end{document}